\documentclass[12pt,leqno,fleqn]{amsart}
\usepackage{amsmath,amstext,amsthm,amssymb,amsxtra}
\usepackage{txfonts,pxfonts} %also txfonts
\usepackage[colorlinks,citecolor=red,pagebackref,hypertexnames=false]{hyperref}
\usepackage[backrefs,author-year]{amsrefs}
\usepackage{pgf,tikz}

\usepackage{mathtools}
\mathtoolsset{showonlyrefs,showmanualtags}

\allowdisplaybreaks
\swapnumbers

\theoremstyle{plain} % definition
\newtheorem{lemma}[equation]{Lemma}
\newtheorem{proposition}[equation]{Proposition}
\newtheorem{theorem}[equation]{Theorem}
\newtheorem{corollary}[equation]{Corollary}

\newtheorem{cww}[equation]{Chang-Wilson-Wolff Inequality }
\newtheorem{2cww}[equation]{Two Parameter Chang-Wilson-Wolff Inequality }
\newtheorem*{roth}{K.~Roth's Theorem}
\newtheorem*{schmidt}{W.~Schmidt's Theorem}
\newtheorem*{jn}{The John-Nirenberg Estimates}

\theoremstyle{definition}
\newtheorem{definition}[equation]{Definition}

\theoremstyle{remark}
\newtheorem{remarks}[equation]{Remarks}
\newtheorem{remark}[equation]{Remark}

\numberwithin{equation}{section}

\def\norm#1.#2.{\lVert#1\rVert_{#2}}
\def\Norm#1.#2.{\bigl\lVert#1\bigr\rVert_{#2}}
\def\NOrm#1.#2.{\Bigl\lVert#1\Bigr\rVert_{#2}}
\def\NORm#1.#2.{\biggl\lVert#1\biggr\rVert_{#2}}
\def\NORM#1.#2.{\Biggl\lVert#1\Biggr\rVert_{#2}}

\def\ip#1,#2,{\langle #1,#2\rangle}
\def\Ip#1,#2,{\bigl\langle#1,#2\bigr\rangle}
\def\IP#1,#2,{\Bigl\langle#1,#2\Bigr\rangle}

\def\abs#1{\lvert#1\rvert}

\def\ABs#1{\biggl\lvert#1\biggr\rvert}

\def\XXint#1#2#3{{\setbox0=\hbox{$#1{#2#3}{\int}$}
     \vcenter{\hbox{$#2#3$}}\kern-.5\wd0}}

\def\rev{\textnormal {rev}}
\def\d{\textnormal {d}}

\begin{document}
%%%%%%%%%%%%%%%%%%%%%%%%%%%%%  Title
\title[Discrepancy Function in Two Dimensions]{Exponential Squared Integrability of the \\
Discrepancy Function in Two Dimensions}
% \subjclass[2000]{Primary: 42B20 Secondary: 42B25, 42B35}
% \keywords{}

\thanks{All authors are grateful to the Fields Institute for hospitality and support, and to
the National Science Foundation for support.}

\author[D. Bilyk]{Dmitriy Bilyk }
\address {School of Mathematics, Institute for Advanced Study, Princeton, NJ 08540, USA.}
\email{bilyk@math.ias.edu}

\author[M.T. Lacey]{Michael T. Lacey}
\address{ School of Mathematics, Georgia Institute of Technology, Atlanta GA 30332, USA.}
\email{lacey@math.gatech.edu}

\author[I. Parissis]{Ioannis Parissis}
\address{Institutionen f\"or Matematik, Kungliga Tekniska H\"ogskolan, SE 100 44, Stockholm, SWEDEN.}
\email{ioannis.parissis@gmail.com}

\author[A. Vagharshakyan]{Armen Vagharshakyan }
\address{ School of Mathematics, Georgia Institute of Technology, Atlanta GA 30332, USA.}
\email{armenv@math.gatech.edu}

\begin{abstract}
Let  $\mathcal A_N$ be an $N$-point set in the unit square and
consider the Discrepancy function
\begin{equation*}
D_N( \vec x) \coloneqq \sharp \big(\mathcal A_N \cap [\vec 0,\vec
x)\big) -N \abs{ [\vec 0,\vec x)},
\end{equation*}
where $ \vec x= (x_1 , x_2)\in [0,1]^2$, $[ 0,\vec x)=\prod _{t=1}
^{2} [0,x_t)$, and $ \lvert[\vec 0,\vec x)  \rvert $ denotes the
Lebesgue measure of the rectangle. We give various refinements of a well-known 
result of \cite{MR0319933} on the $ L ^{\infty } $ norm of $ D_N$.  We show that necessarily
\begin{equation*}
\norm D_N. \operatorname {exp} (L ^{\alpha }). \gtrsim (\log N) ^{
1-1/ \alpha }\,, \qquad 2\le \alpha <\infty \,.
\end{equation*}
The case of $ \alpha = \infty $ is the Theorem of Schmidt. This estimate is sharp.
For the digit-scrambled van der Corput sequence, we have
\begin{equation*}
\norm D_N. \operatorname {exp} (L ^{\alpha }). \lesssim  (\log N) ^{
1-1/ \alpha }\,, \qquad 2\le \alpha <\infty \, ,
\end{equation*}
whenever $N=2^n$ for some positive integer $n$.
This estimate depends upon variants of the Chang-Wilson-Wolff
inequality \cite{MR800004}. We also provide similar estimates for
the $BMO$ norm of $D_N$.
\end{abstract}

\maketitle

%%%%%%%%%%%%%%%%%%%%%%%%%%%%%% SECTION  SECTION SECTION
%%%%%%%%%%%%%%%%%%%%%%%%%%%%%% SECTION  SECTION SECTION
\section{Main Theorems} %\label{s.}

 The common theme of the subject of irregularities of distribution is to
 show that, no matter how $N$ points are selected, their distribution must be far from uniform.
In the present article, we are primarily interested in the precise
behavior of such estimates near the $ L ^{\infty }$ endpoint,
phrased in terms of exponential Orlicz classes. We restrict our
attention to the two-dimensional case.

Let $ \mathcal A_N \subset [0,1] ^{2}$ be a set of  $N$ points in
the unit square. For $\vec x  = (x_1, x_2)\in [0,1]^2$, we define
the \emph{Discrepancy function} associated to $ \mathcal A_N$ as
follows:
\begin{equation*}
D_N (\vec x) \coloneqq \sharp \bigl(\mathcal A_N \cap [0, \vec x)\bigr) - N \abs{ [0, \vec x)}\,,
\end{equation*}
where $ [0,\vec x)$ is the axis-parallel rectangle in the unit
square with one vertex at the origin and the other at $ \vec x =(x_1
,x_2)$, and $ \abs{ [0, \vec x)}=x_1 \cdot x_2$ denotes the Lebesgue
measure of the rectangle. This is the difference between the actual
number of points in the rectangle $ [0,\vec x)$ and the expected
number of points in this rectangle. The relative size of this
function, in various senses, must necessarily increase with $ N$.
The principal result in this direction is due to Roth
\cite{MR0066435}:

\begin{roth} \label{t.roth} In all dimensions $ d\ge 2$, we have the
following estimate
\begin{equation}\label{e.roth}
\norm D_N. 2. \gtrsim (\log N) ^{(d-1)/2}
\end{equation}
where the implied constant is only a function of dimension $ d$.
\end{roth}

The same bound holds for the $ L ^{p}$ norm, for $ 1<p<\infty $,
\cite{MR0491574}, and is known to be sharp as to the order of
magnitude, see \cite{MR610701} and \cite{MR903025} for a history of
this subject (for the case $d=2$, see Corollary \ref{c.p} below).
The endpoint cases of $ p=1$ and $ p=\infty $ are much harder.

We concentrate on the case of $ p=\infty $ in this note, just in
dimension $ d=2$, and refer the reader to
\cites{MR1032337,0705.4619,math.CA/0609815,MR637361}  for more
information about the case of $ d\ge 3$.  For information about the
case of  $ p=1 $, see \cites{MR637361,math.NT/0609817}. As it has
been shown in the fundamental theorem of W. Schmidt
\cite{MR0319933}, in dimension $d=2$, the lower bound on the
$L^\infty$ norm of the Discrepancy function is substantially greater
than the $L^p$ estimate \eqref{e.roth}:

\begin{schmidt}\label{t.schmidt}
For any set $ \mathcal A_N\subset [0,1]^2$ we have
\begin{equation}\label{e.schmidt}
\norm D_N .\infty . \gtrsim \log N \,.
\end{equation}
\end{schmidt}

This theorem is also sharp: one particular example is the famous van
der Corput set \cite{Cor35} --  a detailed discussion is contained
in \S3. In this paper, we give an interpolant between the results of
Roth and Schmidt, which is measured in the scale of exponential
Orlicz classes.

%%%%%%%%%%%%%%%%%%%%%%%%%%%%%% THEOREM THEOREM THEOREM
\begin{theorem}\label{t.lower}
For any $N$-point set $ \mathcal A_N \subset [0,1]^2$ we have
\begin{equation*}
\norm D_N .\operatorname {exp} (L ^{\alpha }) . \gtrsim (\log N ) ^{1-1/\alpha }\,,
\qquad 2\le \alpha < \infty \,.
\end{equation*}
\end{theorem}
%%%%%%%%%%%%%%%%%%%%%%%%%%%%%% THEOREM THEOREM THEOREM

Of course the lower bound of $ (\log N) ^{1/2}$, the case of $
\alpha =2$ above, is a consequence of Roth's bound.  The other
estimates require proof, which is a variant of  Hal{\'a}sz's
argument \cite{MR637361}.  We give details below and also remark
that this estimate in the context of the Small Ball Inequality
\cites{MR95k:60049,MR96c:41052} is known \cite {2000b:60195}. In
addition, we demonstrate that the previous theorem is sharp.

%%%%%%%%%%%%%%%%%%%%%%%%%%%%%% THEOREM THEOREM THEOREM
\begin{theorem}\label{t.vdc} For all $ N$, there is a choice
of $ \mathcal A_N$, specifically the digit-scrambled van der Corput
set (see Definition~\ref{d.vdc}), for which we have
\begin{equation}\label{e.vdc}
\norm D_N. \operatorname {exp} (L ^{\alpha }). \lesssim (\log N) ^{
1-1/ \alpha }\,, \qquad 2\le \alpha <\infty \,.
\end{equation}
\end{theorem}
%%%%%%%%%%%%%%%%%%%%%%%%%%%%%% THEOREM THEOREM THEOREM

In view of Proposition \ref{p.comparable}, taking $\alpha=2$, the
theorem above immediately yields the sharpness of the $L^p$ lower
bounds in $d=2$ with explicit dependence of constants on $p$.

\begin{corollary}\label{c.p} For every $1\le p<\infty$, the set $\mathcal A_N$ from Theorem \ref{t.vdc}
satisfies
\begin{equation}
\norm D_N. p. \lesssim p^{1/2} (\log N )^{1/2},
\end{equation}
where the implied constant is independent of $p$.
\end{corollary}

There is another variant of the Roth lower bound, which we state here. 

%%%%%%%%%%%%%%%%%%%%%%%%%%%%%% THEOREM THEOREM THEOREM
\begin{theorem}\label{t.bmolower}  We have the estimate 
\begin{equation*}
\norm  D_N . \operatorname {BMO} _{1,2}. \gtrsim ( {\log
N})^{1/2}\,,
\end{equation*}
where the norm is the dyadic Chang-Fefferman product $\operatorname {BMO}$
norm (see Definition~\ref{d.cf}), introduced in \cite{MR584078}.
\end{theorem}
%%%%%%%%%%%%%%%%%%%%%%%%%%%%%% THEOREM THEOREM THEOREM

Indeed, this Theorem is just a corollary to a standard proof of Roth's Theorem, 
and its main
interest lies in the fact that the estimate above is sharp.  It is useful to 
recall the simple observation that the  $ \operatorname {BMO}$ norm is insensitive to functions that 
are constant in either the vertical or horizontal direction.  That is, we have 
$ \norm  D_N . \operatorname {BMO} _{1,2}.= \norm  \widetilde D_N . \operatorname {BMO} _{1,2}. $, 
where 
\begin{equation*}
\begin{split}
\widetilde D_N (x_1,x_2)
&= D_N (x_1,x_2)- \int _0 ^{1} D_N (x_1,x_2) \; dx_1
\\& \qquad -\int _0 ^{1} D_N (x_1,x_2) \; dx_2 +
\int _0 ^{1 }\!\!\int _0 ^{1} D_N (x_1,x_2) \; dx_1 \, dx_2\,.
\end{split}
\end{equation*}

%%%%%%%%%%%%%%%%%%%%%%%%%%%%%% THEOREM THEOREM THEOREM
\begin{theorem}\label{t.bmo}
For $ N= 2 ^{n}$, there is a choice of $ \mathcal A_N$, specifically
the digit-scrambled van der Corput set, for which we have
\begin{equation}\label{e.bmoupper}
\norm   D_N . \operatorname {BMO} _{1,2} . \lesssim (\log
N) ^{1/2} \,.
\end{equation}
\end{theorem}
%%%%%%%%%%%%%%%%%%%%%%%%%%%%%% THEOREM THEOREM THEOREM

The main point of these results is that they unify the theorems of
Roth and Schmidt in a sharp fashion.  This line of research is also
of interest in higher dimensions, but the relevant conjectures do
not seem to be as readily apparent.  As such, we think that this is
an interesting theme for further investigation.

In the next section we collect a variety of results needed to prove
the main Theorems. These results are drawn from the theory of
Irregularities of Distribution, Harmonic Analysis, Probability
Theory and other subjects.   In \S3 we discuss the structure of the
digit-scrambled van der Corput set. Section 4 is dedicated to the
analysis of the Haar decomposition of the Discrepancy function for
the van der Corput set. The proofs of the main theorems above are
then taken up in the \S5 and \S6.

The results of this paper concern refinements of the $ L ^{\infty }$-endpoint 
estimates for the Discrepancy Function.  In three dimensions, even the 
correct form of Schmidt's Theorem is not yet known, making the discussion of these 
results in three dimensions entirely premature, though speculation about 
such results could inform the analysis of the more difficult three dimensional case.  
See \cites{math.CA/0609815,0705.4619} for recent information about the higher dimensional 
versions of Schmidt's Theorem. 

The authors thank the referee for an expert reading, and suggestions to improve the paper.

%%%%%%%%%%%%%%%%%%%%%%%%%%%%%% SECTION  SECTION SECTION
%%%%%%%%%%%%%%%%%%%%%%%%%%%%%% SECTION  SECTION SECTION
\section{Preliminary Facts} %\label{s.}

We suppress many constants which do not affect the arguments in
essential ways. $ A \lesssim B$ means that there is an absolute
constant $ K>0$ such that $ A \le K B$. Thus $ A \lesssim 1$ means
that $ A$ is bounded by an absolute constant. And if $ A \lesssim B
\lesssim A$, we write $ A \simeq B$.

\bigskip

%%%%%%%%%%%%%%%%%%%%%%%%%%%%%% SUBSECTION SUBSECTION SUBSECTION SUBSECTION
 %%%%%%%%%%%%%%%%%%%%%%%%%%%%%% SUBSECTION SUBSECTION SUBSECTION SUBSECTION
\subsection*{Inequalities}%\label{ss.}

We recall the square function inequalities for martingales, in a
form convenient for us.

In one dimension, the class of dyadic intervals in the unit interval  are
$\mathcal D {} \coloneqq {}\{ [j2^{-k},(j+1)2^{-k})
\mid j,k\in \mathbb N\,, 0\le j < 2 ^{k}\} $.  Let $ \mathcal D_n$ denote the
dyadic intervals of length $ 2 ^{-n}$, and by abuse of notation, also the
sigma field generated by these intervals.  For an integrable function $ f$ on
$ [0,1]$, the conditional expectation is
\begin{equation*}
f_n=\mathbb E (f \mid \mathcal D_n) \coloneqq  \sum _{I\in \mathcal D_n} \mathbf 1_{I}  \cdot
\lvert  I\rvert ^{-1}  \int _{I} f (y)\;dy\,.
\end{equation*}
The sequence of functions $ \{ f_n \mid n\ge 0\}$ is a \emph{martingale}.  The
\emph{martingale difference sequence } is $ d_0=f_0$, and $ d_n= f_n-f _{n-1}$ for
$ n\ge 1$.  The sequence of functions $ \{d_n\mid n\ge 0\}$ are pairwise orthogonal.
The \emph{square function} is
\begin{equation*}
\operatorname S (f) \coloneqq \Biggl[\sum _{n=0} ^{\infty } \lvert  d_n\rvert ^2   \Biggr]
^{1/2} \,.
\end{equation*}
We have the following extension of the Khintchine inequalities.

%%%%%%%%%%%%%%%%%%%%%%%%%%%%%% THEOREM THEOREM THEOREM
\begin{theorem}\label{t.martKhintchine} The inequalities below hold, for
some absolute choice of constant $ C>0$.
\begin{equation}\label{e.martKhintchine}
\norm f. p. \le C \sqrt p \norm \operatorname S (f).p.\,, \qquad 2\le p < \infty \,.
\end{equation}
In addition, this inequality holds for Hilbert space valued
functions $ f$.
\end{theorem}
%%%%%%%%%%%%%%%%%%%%%%%%%%%%%% THEOREM THEOREM THEOREM

For real-valued martingales, this was observed by \cite{MR800004}.
The extension to Hilbert space valued martingales is useful for us
and is proved in \cite{MR1439553}. The best constants in these
inequalities are known for $ p\ge 3$ \cite{MR1018577}.

%%%%%%%%%%%%%%%%%%%%%%%%%%%%%% SUBSECTION SUBSECTION SUBSECTION SUBSECTION
 %%%%%%%%%%%%%%%%%%%%%%%%%%%%%% SUBSECTION SUBSECTION SUBSECTION SUBSECTION
\subsection*{Orlicz Spaces}%\label{ss.}

For background on Orlicz Spaces, we refer the reader to \cite{MR0500056}.
Consider a symmetric convex function $ \psi $,
which is zero at the origin,  and is otherwise non-zero.  Let $ (\Omega , P)$
be a probability space, on which our functions are defined, and let  $ \mathbb E $
denote expectation over the probability space.
We can define
\begin{equation}\label{e.psi}
\norm f. L ^\psi . = \inf \{ K>0\mid  \mathbb E \psi (f \cdot K ^{-1} )\le 1\}\,,
\end{equation}
where we define the infimum over the empty set to be $ \infty $. The
set of functions $ L ^{\psi } = \{f \mid  \norm f. L ^{\Psi } . <
\infty \}$ is a normed linear space, called the Orlicz space
associated with $ \psi $.

We are interested in, for instance, $ \psi (x)= \operatorname e ^{x
^2 }-1$, in which case we denote the Orlicz space by $ \operatorname
{exp} (L ^{2})$. More generally, for $\alpha >0$, we let $ \psi
_{\alpha } (x)$ be a symmetric convex function which equals $
\operatorname e ^{\lvert  x\rvert ^{\alpha } }-1$ for $ \lvert
x\rvert $ sufficiently large, depending upon $ \alpha $.\footnote{We
are only interested in measuring the behavior of functions for large
values of $ f$, so this requirement is sufficient. For $ \alpha >1$,
we can insist upon this equality for all $ x$.}  And we write $ L
^{\psi _{\alpha }} = \operatorname {exp}(L ^{\alpha })$. These are
the spaces used in the statements of our main Theorems \ref{t.lower}
and \ref{t.vdc}. It is obvious that, for all $1\le p<\infty$ and
$\alpha
>0$, we have $L^p \supset \operatorname {exp}(L ^{\alpha }) \supset
L^\infty$, hence Theorem \ref{t.lower} can be indeed viewed as
interpolation between the estimates of Roth \eqref{e.roth} and
Schmidt \eqref{e.schmidt}. The following useful proposition is
well-known and follows from elementary methods.

%%%%%%%%%%%%%%%%%%%%%%%%%%%%%% PROPOSITION PROPOSITION PROPOSITION
\begin{proposition}\label{p.comparable} We have the following equivalence of norms
valid for all $ \alpha >0$:
\begin{equation*}
\norm f. \operatorname {exp}(L ^{\alpha }). \simeq
\sup _{p>1} p ^{-1/\alpha } \norm f. p. \,.
\end{equation*}
\end{proposition}
%%%%%%%%%%%%%%%%%%%%%%%%%%%%%% PROPOSITION PROPOSITION PROPOSITION

We shall also make use of the duality relations for the exponential
Orlicz classes. For $ \alpha >0$, let $ \varphi _{\alpha } (x)$ be a
symmetric convex function which equals $ \lvert x\rvert (\log
(3+\lvert x\rvert) ) ^{\alpha } $ for $ \lvert x\rvert $
sufficiently large, depending upon $ \alpha $.\footnote{For $ \alpha
\ge 1$, we can take this as the definition for all $ \lvert
x\rvert\ge 0 $.} The Orlicz space $ L ^{\varphi _{\alpha }}$ is
denoted as $ L ^{\varphi _{\alpha }} = L (\log L) ^{\alpha }$. The
propositions below are standard.

%%%%%%%%%%%%%%%%%%%%%%%%%%%%%% PROPOSITION PROPOSITION PROPOSITION
\begin{proposition}\label{p.dual} For $ 0<\alpha <\infty $,  the two
Orlicz spaces $ \operatorname {exp}(L ^{\alpha })$ and $ L (\log L)
^{1/\alpha }$ are Banach spaces which are dual to one another.
\end{proposition}
%%%%%%%%%%%%%%%%%%%%%%%%%%%%%% PROPOSITION PROPOSITION PROPOSITION

%In particular, the dual to $ L (\log L)$ is $ \operatorname {exp}(L)$, and
%the dual to $ \operatorname {exp}( L ^{2} )$ is $ L \sqrt {\log L}$.
%We omit the proof of the next Proposition.

%%%%%%%%%%%%%%%%%%%%%%%%%%%%%% PROPOSITION PROPOSITION PROPOSITION
\begin{proposition}\label{p.indicator}  Let $ E$ be a measurable subset of
a probability set.  We have
\begin{equation*}
\norm \mathbf 1_{E} . L (\log L) ^{1/\alpha }.  \simeq \mathbb P (E)
\cdot (1 - \log \mathbb P (E) ) ^{1/\alpha }\,.
\end{equation*}
\end{proposition}
%%%%%%%%%%%%%%%%%%%%%%%%%%%%%% PROPOSITION PROPOSITION PROPOSITION

%%%%%%%%%%%%%%%%%%%%%%%%%%%%%% SUBSECTION SUBSECTION SUBSECTION SUBSECTION
 %%%%%%%%%%%%%%%%%%%%%%%%%%%%%% SUBSECTION SUBSECTION SUBSECTION SUBSECTION
\subsection*{Chang-Wilson-Wolff Inequality}%\label{ss.}

 Each dyadic interval has a left and right half, $ {I_{\textup{left}}}, {I_{\textup{right}}}$
respectively,  which are also dyadic.  Define the
 Haar function associated with $ I$ by
 \begin{equation*}
h_I \coloneqq -\mathbf 1 _{I_{\textup{left}}}+ \mathbf 1 _{I_{\textup{right}}}
\end{equation*}
 Note that here the Haar functions are normalized in $ L ^{\infty }$.
 In particular, the square function with this normalization has the
 form
 \begin{equation*}
\operatorname S (f) ^2 = \sum _{I\in \mathcal D} \frac {  \ip f,
h_I, ^2 } {\lvert  I\rvert ^2  } \mathbf 1_{I} \,, \qquad\textup{for
}\quad f (x)= \sum _{I} \frac {\ip f, h_I, } {\lvert  I\rvert } h_I
(x).
\end{equation*}

We can now deduce the Chang-Wilson-Wolff inequality.

\begin{cww} For all Hilbert space valued martingales, we have
\begin{equation*}
\norm f . \operatorname {exp}(L ^2 ). \lesssim \norm \operatorname S (f). \infty . \,.
\end{equation*}
\end{cww}

Indeed, we have
\begin{equation*}
\norm f. p. \lesssim \sqrt p \cdot \norm \operatorname S (f).p.
\lesssim \sqrt p \cdot \norm \operatorname S (f). \infty . \,.
\end{equation*}
Taking $ p\to \infty $, and using Proposition~\ref{p.comparable}, we deduce the
inequality above.

\medskip

 In dimension $2 $, a \emph{dyadic rectangle} is a product of dyadic intervals, thus an
element of
 $\mathcal D^2 $.   A Haar function associated to
 $R $ is the product of the Haar functions associated
 with each side of $R $, namely  for $ R_1\times   R_2$,
 \begin{equation*}
 h_{R_1\times R_2 }(x_1 ,x_2) {} \coloneqq \prod _{t=1} ^{2}h _{R_t}(x_t)\,.
 \end{equation*}
 See Figure~\ref{f.haar}.
 Below, we will expand  the definition of Haar functions, so that we can describe
a basis for $ L ^{2} ([0,1] ^2 )$.

%\begin{figure}
%\begin{tikzpicture}
%  %%%%\begin{scope}[ fill opacity=0.5]

%\draw (2.5,2.5) node {$ h_R$};
%\draw (0,0) rectangle (8,2);
%\draw[fill,fill opacity=0.5,lightgray] (0,0) rectangle (4,1);
%\draw[fill,fill opacity=0.5,lightgray] (4,1) rectangle (8,2);

%\draw (6,4.5) node {$ h_S$};
%\draw (4,0) rectangle (8,4);
%\draw[fill,fill opacity=0.5,lightgray] (4,0) rectangle (6,2);
%\draw[fill,fill opacity=0.5,lightgray] (6,2) rectangle (8,4);
%  %%%%\end{scope}
%\end{tikzpicture}
%\caption{Two Haar functions. }
%\label{f.haar}
%\end{figure}

We will concentrate on rectangles of a fixed volume, contained in $[0,1]^2 $. The notion of the square function is also useful in the two dimensional
  context.  It has the form
   \begin{equation} \label{e.haar2}
\operatorname S (f) ^2 = \sum _{R\in \mathcal D ^2} \frac { \ip f
,h_R , ^2 } {\lvert  R\rvert ^2  } \mathbf 1_{R} \,, \qquad \textup
{ for } \quad f (x)= \sum _{R\in \mathcal D ^2 }  \frac { \ip f ,h_R
,} { {\lvert R\rvert }} h_R (x)\,.
\end{equation}
Jill Pipher \cite{MR850744} observed the following extension of the Chang-Wilson-Wolff inequality.

\begin{2cww}  For functions $ f$ in the plane as in \eqref{e.haar2} we have
\begin{equation*}
\norm f . \operatorname {exp} (L) .  \lesssim \norm \operatorname S (f) . \infty . \,.
\end{equation*}
\end{2cww}

Namely, in the case of two-parameters, the exponential integrability
has been reduced by a factor of two.  This follows from a two-fold
application of the Littlewood-Paley inequalities, with best
constants, for Hilbert space valued functions.  Details can be found
in \cites{MR850744,MR1439553,math.CA/0609815}. In fact, we will need
the following variant.

%%%%%%%%%%%%%%%%%%%%%%%%%%%%%% THEOREM THEOREM THEOREM
\begin{theorem}\label{t.one}  Let $ n\ge 1$ be an integer.
Suppose that $ f$ on the plane has the expansion
\begin{equation*}
f = \sum _{\substack{R\in \mathcal D ^2\\ \lvert  R\rvert = 2 ^{-n}  } }
\frac { \ip f ,h_R ,} {{\lvert  R\rvert }}  h_R \,.
\end{equation*}
That is, $ f$ is in the linear span of Haar functions with a fixed volume.  Then, we
have the estimate
\begin{equation*}
\norm f . \operatorname {exp} (L ^2 ) . \lesssim  \norm S (f) . \infty . \,.
\end{equation*}
\end{theorem}
%%%%%%%%%%%%%%%%%%%%%%%%%%%%%% THEOREM THEOREM THEOREM

Thus, if $ f$ is in the linear span of a `one-parameter' family of
rectangles, we regain the exponential-squared integrability. The
proof is straightforward.  As the volumes of the rectangles are
fixed, one need only apply the one-parameter Chang-Wilson-Wolff
inequality in, say, the $ x_1$ variable, holding the $ x_2$ variable
fixed.

The following simple proposition reduces the proof of Theorem
\ref{t.vdc} to the case $\alpha=2$.
%%%%%%%%%%%%%%%%%%%%%%%%%%%%%% PROPOSITION PROPOSITION PROPOSITION
\begin{proposition}\label{p.AA}  Suppose that for $ A\ge 1$, we have
\begin{equation*}
\norm f. \operatorname {exp}( L ^2)  . \le \sqrt A\,, \qquad \norm f. \infty . \le A \,.
\end{equation*}
It follows that
\begin{equation*}
\norm f. \operatorname {exp} (L ^{\alpha }). \le A ^{ 1- 1/ \alpha }
\,, \qquad 2\le \alpha < \infty \,.
\end{equation*}
\end{proposition}
%%%%%%%%%%%%%%%%%%%%%%%%%%%%%% PROPOSITION PROPOSITION PROPOSITION
%%%%%%%%%%%%%%%%%%%%%%%%%%%%%% SUBSECTION SUBSECTION SUBSECTION SUBSECTION
 %%%%%%%%%%%%%%%%%%%%%%%%%%%%%% SUBSECTION SUBSECTION SUBSECTION SUBSECTION
\subsection*{Bounded Mean Oscillation}%\label{ss.}

We recall facts about dyadic $BMO$ spaces, see \cites{cf1,MR584078}.

 We need to subtract some terms from $D_N$, as it
is not necessarily in the span of the Haar functions as we have
defined them. The deficiency is that standard Haar functions on the
unit square have zero means in both directions. Hence, for a dyadic
interval $ I\in \mathcal D$, we also need to consider
\begin{equation*}
h ^{1}_I = \mathbf 1_{I} = \lvert  h_I\rvert\,.
\end{equation*}
And set $ h^0_I=h_I$, where `$ 0$' stands for `zero integral' and `$
1$' for `non-zero integral.' In the plane, for $ \epsilon _1\,, \,
\epsilon _2\in \{0,1\}$ set
\begin{equation} \label{e.01}
h_{R_1 \times R_2} ^{\epsilon _1, \epsilon _2}(x_1, x_2) = \prod _{j=1} ^{2}
h ^{\epsilon _j} _{R_j} (x_j)\,.
\end{equation}
We will sometimes write $h_R=h_R ^{0,0}$ in order to simplify our
notation. With these definitions we have the following
{\emph{orthogonal}} basis for $ L ^2 ([0,1] ^2 )$.
\begin{equation*}
\{ h ^{1,1} _{[0,1] ^2 } \}
\cup
\{ h ^{1,0} _{[0,1] \times I}\,,\,
h ^{1,0} _{I\times [0,1] }\mid I\in \mathcal D\}
\cup
\{h_R ^{0,0} \mid R\in \mathcal D ^2 \}\,.
\end{equation*}
%Thus, we have adequately accounted for all the $ 1$'s we need to obtain a basis. We will sometimes write $h_R=h_R ^{0,0}$ in order to simplify our notation.
There are  couple of different $ \operatorname {BMO}$ spaces that
are relevant here. Let us begin with the variants of the more
familiar C.~Fefferman, one-parameter, dyadic $ \operatorname {BMO}$
spaces.

%%%%%%%%%%%%%%%%%%%%%%%%%%%%%%  DEFINITION DEFINITION DEFINITION
\begin{definition}\label{d.f}  Define the space $ \operatorname {BMO}_1$
to be those square integrable functions $ f$ in the span of $\{ h
^{0,1} _{I \times [0,1]} \mid I \in \mathcal D \}$ which satisfy
\begin{equation}\label{e.1}
\norm f. \operatorname {BMO}_1 . \coloneqq
\sup _{J\in \mathcal D} \Bigl[ \lvert  J\rvert ^{-1} \sum _{\substack{I\in \mathcal D\\ I\subset J}}
\frac {\ip f ,h ^{0,1} _{I \times [0,1]} , ^2  } {\lvert  I\rvert }
 \Bigr] ^{1/2} < \infty \,.
\end{equation}
Define $ \operatorname {BMO}_2$ similarly, with the roles of the first and second coordinate
reversed.
\end{definition}
%%%%%%%%%%%%%%%%%%%%%%%%%%%%%%  DEFINITION DEFINITION DEFINITION

%%%%%%%%%%%%%%%%%%%%%%%%%%%%%%  DEFINITION DEFINITION DEFINITION
\begin{definition}\label{d.cf}  Dyadic Chang-Fefferman $ \operatorname {BMO}_{1,2}$ is defined to be
those square integrable functions $ f$ in the linear span of $ \{h_R \mid R\in \mathcal D ^2
\}$, for which we have
\begin{equation}\label{e.12}
\norm f. \operatorname {BMO}_{1,2} . \coloneqq \sup _{ U\subset
[0,1] ^2 } \Bigl[ \lvert  U\rvert ^{-1} \sum _{\substack{R \in
\mathcal D ^2
\\ R\subset U}}
\frac {\ip f ,h _R , ^2  } {\lvert  R\rvert }
 \Bigr] ^{1/2} < \infty \,.
\end{equation}
We stress that the supremum is over {\em all} measurable subsets $
U\subset [0,1] ^2 $, not just rectangles.
\end{definition}
%%%%%%%%%%%%%%%%%%%%%%%%%%%%%%  DEFINITION DEFINITION DEFINITION

It is well-known that these `uniform square integrability'
conditions imply that the corresponding functions enjoy higher
moments.  This is usually phrased as the John-Nirenberg
inequalities, which we state here in their sharp exponential form.

\begin{jn}  We have the following estimate for $ f\in \operatorname {BMO} _{1}$,
and $ \varphi \in \operatorname {BMO} _{1,2}$.
\begin{align}\label{e.jn1}
\norm f. \operatorname {exp} (L). &\lesssim \norm f . \operatorname {BMO} _{1}.
\\ \label{e.jn12}
\norm \varphi . \operatorname {exp} (\sqrt L). &\lesssim \norm \varphi  . \operatorname {BMO} _{1,2}.
\end{align}
%The corresponding inequality for $ \operatorname {BMO} _{2}$, namely \eqref{e.jn1}, with
%the role of the first and second coordinates reversed, also holds.
\end{jn}

Note that in the second inequality, \eqref{e.jn12}, the number of
parameters has doubled, hence the exponential integrability has
decreased by a factor of two. Of course, if the square function of $ f$ is bounded, one sees
immediately that the functions are necessarily in $ BMO$.  And in
this circumstance  the Chang-Wilson-Wolff inequalities give an
essential strengthening of the
 John-Nirenberg  estimates.

\bigskip

%%%%%%%%%%%%%%%%%%%%%%%%%%%%%% SUBSECTION SUBSECTION SUBSECTION SUBSECTION
 %%%%%%%%%%%%%%%%%%%%%%%%%%%%%% SUBSECTION SUBSECTION SUBSECTION SUBSECTION
\subsection*{Discrepancy}%\label{ss.}

Below, we will refer to the two parts of the Discrepancy function as
the `linear' and the `counting' part.  Specifically, they are
\begin{align}\label{e.linear}
L_N (\vec x) &= N x_1 \cdot x_2 \,,
\\  \label{e.counting}
C_ {\mathcal P} (\vec x) &= \sum _{\vec p \in \mathcal P} \mathbf 1_{[\vec p, \vec 1) } (\vec x) \,.
\end{align}
Here, $ \mathcal P$ is the subset of the unit square of cardinality
$ N$. %We will sometimes write $D_N(\mathcal P)$ when we want to
%emphasize the dependence on the point distribution $\mathcal P$ as
%is the case when we seek upper bounds for the discrepancy function
%in terms of specific constructions of sets $\mathcal P$.
In proving upper bounds on the Discrepancy function, one of course
needs to capture a cancellation between these two, that is large
enough to nearly completely cancel the nominal normalization by $
N$.

We recall some definitions and facts about Discrepancy
which  are well represented in the literature, and apply to general selection of
point sets, see
\cites{MR0066435,MR554923,MR903025}.

We call a function $f$ an \emph{$\mathsf r$ function  with parameter
$ \vec r= (r_1 , r_2)$} if $ \vec r\in \mathbb N ^{2}$, and
\begin{equation}	
\label{e.rfunction} f=\sum_{ R \in \mathcal R _{\vec r}}
\varepsilon_R\, h_R\,,\qquad \varepsilon_R\in \{\pm1\}\, ,
\end{equation}
where we set $ \mathcal R _{\vec r} \coloneqq \{ R= R_1 \times R_2
\mid R\in \mathcal D^2\,, R\subset [0,1]^2\, ,\ \lvert  R_t\rvert= 2
^{- r_t}\,, \ t=1,2\}\,.$
 We will use $f _{\vec r} $ to denote a
generic $\mathsf r$ function. A fact used without further comment is
that $ f _{\vec r} ^2 \equiv 1$.

Let $ \abs{ \vec r}= \sum _{t=1} ^{2} r_t=n$, which we refer to as
the index of the $ \mathsf r$ function. And let $\mathbb H _n ^{2}
\coloneqq \{\vec r \in \{0,1 ,\dotsc, n\} ^2
 \mid \abs{ \vec r}=n\}$, i.e., the set of all $\vec r \,$'s such that
 rectangles in $\mathcal R_{\vec r}$ have  area $2^{-n}$.
It is fundamental to the subject that $\sharp \mathbb H _n ^{2}  =
n+1 $. We refer to  $ \{f _{\vec r} \mid r\in \mathbb H _n ^{2}\}$
as hyperbolic $ \mathsf r$ functions. The next four Propositions are
standard.

%%%%%%%%%%%%%%%%%%%  Lemma
\begin{proposition}\label{p.rvec}
For any selection  $ \mathcal A_N$ of $ N$ points in the unit cube
the following holds. Fix $ n$ with $ 2N< 2 ^{n}\le 4N$. For each
$\vec r\in \mathbb H_{n} ^{2}$, there is an $\mathsf r$ function $f
_{\vec r} $ with
\begin{equation*}
\ip D_N, f _{\vec r}, \gtrsim 1\,.
\end{equation*}
\end{proposition}

%%%%%%%%%%%%%%%%%%%%%%%%%%%%%% PROOF PROOF PROOF
 \begin{proof}
 There is a very elementary one dimensional fact: for all dyadic intervals $I$,
 \begin{equation} \label{e.veryelementary}
 \int _{0}  ^{1 } x   \cdot   h _{I}(x) \; dx =\tfrac 14 \abs{ I} ^2  \,.
 \end{equation}
 This immediately implies that
 \begin{equation} \label{e.vv}
 \langle  x_1\cdot x_2\, , \,h_R ^{0,0}(x_1, x_2) \rangle = 4 ^ {-2}\abs{ R} ^2  \,.
 \end{equation}
 Thus, the inner product with the linear part of the Discrepancy function is completely
 straightforward. We have $ \ip L , h_R ^{0,0}, \ge 4 ^{-2} N \lvert  R\rvert ^2
 \ge 4 \lvert  R\rvert $ for $R \in \mathcal R_{\vec r}$ with $\vec r \in \mathbb H_n^2$.

 %  Recall that $\mathcal A_N$, the distribution of
 %$N$ points in the unit cube, is fixed.
 Call a rectangle $R\in \mathcal R _{\vec r}$
 \emph{good} if $R$ does {\bf not} intersect $\mathcal A_N$, otherwise call it \emph{bad}.
  Set
  \begin{equation}  \label{e.f_r}
  f _{\vec r} {} \coloneqq \sum _{R\in \mathcal R _{\vec r}}
  \operatorname {sgn} (\ip D_N, h_R,) h_R \,.
  \end{equation}
  Each bad rectangle contains at least one point in $\mathcal A_N$, and  $2^ n\ge2N$, so
  there are at least $N$ good rectangles.  Moreover, one should observe that the counting function
 $\sharp (\mathcal A_N\cap [0,\vec x))$ is orthogonal to $ h_R$ for each good rectangle $
 R$.  That is,
\begin{equation*}
\ip C_ {\mathcal A_N}, h ^{0,0}_R, =0\,, \qquad
\textup{whenever}\quad R\cap \mathcal A_N=\emptyset\,.
\end{equation*}
Critical to this property is the fact that Haar functions have mean
zero on each line parallel to the coordinate axes.

 Thus, by \eqref{e.vv}, for a {\em good} rectangle $R\in \mathcal R _{\vec
 r}$ we have
 \begin{equation*}
 \ip D_N, h_R,=-\ip L_N, h_R,=-N\ip   \abs{ [0,\vec x)} , h_R(\vec x),=-N2 ^ {-2n-4}
\lesssim - 2  ^{-n}\,.
 \end{equation*}
 Hence, to complete the proof, we can estimate
 \begin{align*}
 \ip D_N, f _{\vec r}, \ge
 \sum _{\substack{R\in \mathcal R _{\vec r}\\ \text {$R$ is good} } }
 \lvert \ip D_N, h_R, \rvert \gtrsim 2^  {-n} \sharp \{
 R\in \mathcal R _{\vec r}\mid \text {$R$ is good} \}
 \gtrsim 1\,.
 \end{align*}

 \end{proof}
% %%%%%%%%%%%%%%%%%%%%%%%%%%%%%% PROOF PROOF PROOF

%%%%%%%%%%%%%%%%%%%%%%%%%%%%%% PROPOSITION PROPOSITION PROPOSITION
\begin{proposition}\label{p.>n} Let $ f _{\vec s}$ be any $ \mathsf r$ function
with $ \abs{ \vec s}>n$.  We have
\begin{equation*}
\abs{ \ip D_N, f _{\vec s}, } \lesssim N {2 ^{-\abs{ \vec s}}}\,.
\end{equation*}

\end{proposition}
%%%%%%%%%%%%%%%%%%%%%%%%%%%%%% PROPOSITION PROPOSITION PROPOSITION

% %%%%%%%%%%%%%%%%%%%%%%%%%%%%%% PROOF PROOF PROOF
 \begin{proof}
 This is a brute force proof.  Consider the linear part of the Discrepancy function.
 By (\ref{e.veryelementary}), we have
 \begin{equation*}
 \abs{ \ip  L_N, f _{\vec s}, }\lesssim N 2 ^{-\abs{ \vec s}} \,,
 \end{equation*}
 as claimed.

 Consider the part of the Discrepancy function that arises from the point set.
 Observe that for any point $\vec x_0$ in the point set, we have
 \begin{equation*}
 \abs{ \ip \mathbf 1 _{[\vec 0, \vec x_0)} , f _{\vec s}, } \lesssim 2 ^{- \abs{ \vec
 s}}\,.
 \end{equation*}
 Indeed, of the different Haar functions that contribute to $ f _{\vec s}$, there
 is at most one with non zero inner product with the function
 $\mathbf 1 _{[\vec 0, \vec x)} (\vec x_0) $ as a function of $ \vec x$.  It is the
 one rectangle which contains $ x_0$ in its interior. Thus the inequality above follows.
 Summing it over the $ N$ points in the point set completes the proof of the Proposition.
 \end{proof}
% %%%%%%%%%%%%%%%%%%%%%%%%%%%%%% PROOF PROOF PROOF
%%%%%%%%%%%%%%%%%%%%%%%%%%%%%% PROPOSITION PROPOSITION PROPOSITION
\begin{proposition}\label{p.Not2Many}
In dimension $ d=2$ the following holds. Fix a collection of $
\mathsf r$ functions $\{ f _{\vec r} \mid \vec r \in \mathbb H_{n}
^{2} \}$. Fix an integer $ 2\le v \le n$ and $ \vec s $ with $0\leq s_1,s_2 \leq n$  and $\abs{
\vec s}\ge n+ v -1$. Let $ \operatorname {Count} (\vec s ; v) $ be
{the number of ways to  choose distinct $ \vec r_1 ,\dotsc, \vec r_v\in
\mathbb H _n ^2 $  so that $ \prod _{w=1} ^{v} f _{\vec r_w}$ is an
$ \vec s$ function.} We have

\begin{equation} \label{e.Not2Many}
\operatorname {Count} (\vec s ; v) =  {\abs {\vec s}-n-1 \choose v
-2 }\,.
\end{equation}
\end{proposition}
%%%%%%%%%%%%%%%%%%%%%%%%%%%%%% PROPOSITION PROPOSITION PROPOSITION

%%%%%%%%%%%%%%%%%%%%%%%%%%%%%% PROOF PROOF PROOF
 \begin{proof}
 Fix a vector $ \vec s$ with $ \abs{ \vec s}>n$, and suppose that
 \begin{equation*}
 \prod _{w=1} ^{v} f _{\vec r_w}
 \end{equation*}
 is an $ \vec s$ function.  Then, the maximum of the first coordinates of the $ \vec r_w$
 must be $ s_1$, and similarly for the second coordinate.  Thus, the vector $ s$
 completely specifies two of the $ \vec r_w$.

 The remaining $ v-2$ vectors must be distinct, and take values in the first
 coordinate that are greater than $n-s_2$ and less than $ s_1$.
 Hence there are at most $ \abs{ \vec s}-n-1$
 possible choices for these vectors. %(the first coordinate must be smaller than $s_1$ and greater than $n-s_2$), and from them we select $ v-2$.
 This completes
 the proof.

 \end{proof}
% %%%%%%%%%%%%%%%%%%%%%%%%%%%%%% PROOF PROOF PROOF
In two dimensions,  the decisive product rule holds.  If  $R, R' \in
\mathcal D^2$ are distinct, have the same area and non-empty intersection, then we have
%It is as follows, and we omit the
%proof.
\begin{equation}\label{e.productRule}
h_R \cdot h_{R'} = \pm h_{R\cap R'}.
\end{equation}
This rule is illustrated in Figure~\ref{f.haar} and can be generalized as follows.%, with the
%overlapping light gray rectangles representing $ -1$, so that the
%smaller darker gray rectangle will have sign $ 1$ in the product.

%%%%%%%%%%%%%%%%%%%%%%%%%%%%%% PROPOSITION PROPOSITION PROPOSITION
\begin{proposition}\label{p.productRule}
In dimension $ d=2$ the following holds. Let $ \vec r_1 ,\dotsc,
\vec r_k$ be   elements of $ \mathbb H _n ^{2}$ where one of the
vectors occurs an odd number of times. Then, the product $ \prod
_{j=1} ^{k} f _{\vec r}$ is also an $ \mathsf r$ function. If the $
\vec r_j$ are distinct and $ k\ge 2$, the product has index larger
than $ n$.
\end{proposition}
%%%%%%%%%%%%%%%%%%%%%%%%%%%%%% PROPOSITION PROPOSITION PROPOSITION
%We would like to stress that the failure of the product rule in
%higher dimensions yields a  major complication to the theory.
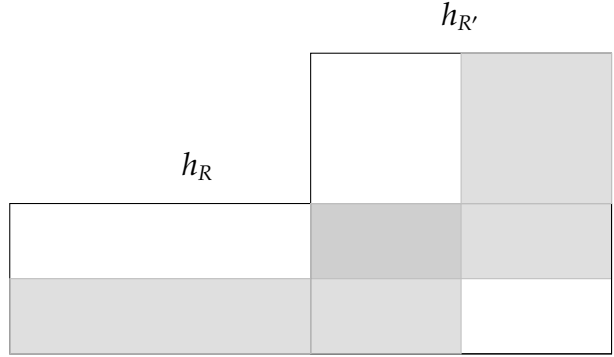
\begin{figure}
\begin{tikzpicture}
%\begin{scope}[ fill opacity=0.5]

\draw (2.5,2.5) node {$ h_R$}; \draw (0,0) rectangle (8,2);
\draw[fill,fill opacity=0.5,lightgray] (0,0) rectangle (4,1);
\draw[fill,fill opacity=0.5,lightgray] (4,1) rectangle (8,2);

\draw (6,4.5) node {$ h_{R'}$}; \draw (4,0) rectangle (8,4);
\draw[fill,fill opacity=0.5,lightgray] (4,0) rectangle (6,2);
\draw[fill,fill opacity=0.5,lightgray] (6,2) rectangle (8,4);
%\end{scope}
\end{tikzpicture}
\caption{Two Haar functions. } \label{f.haar}
\end{figure}

%%%%%%%%%%%%%%%%%%%%%%%%%%%%%% SECTION  SECTION SECTION
%%%%%%%%%%%%%%%%%%%%%%%%%%%%%% SECTION  SECTION SECTION
\section{The Digit-Scrambled van der Corput Set} \label{s.vandercorput}
In this section we introduce the digit-scrambled van der Corput set,
that is, a variation of the classical van der Corput set described,
e.g., in \cite{MR1697825}*{Section 2.1}, and prove some auxiliary
lemmas that will help us exploit its properties. This set will be
our main construction for the upper bounds in Theorems \ref{t.vdc}
and \ref{t.bmo}, although strictly speaking, Theorem \ref{t.bmo} is
satisfied by the standard van der Corput point distribution. The
reasons we need this modified version of the van der Corput set will
become clear by the end of this section.

First, we introduce some additional definitions and notations.

\begin{definition}
For $x\in [0,1)$ define $\d_i(x)$ to be the $i$'th digit in the
binary expansion of $x$, that is
$$\d_i(x)=\lfloor 2^ix\rfloor \space \textup{ mod 2}. $$
\end{definition}

\begin{definition}
For $x\in[0,1)$ we define the \emph{digit reversal} function by means of the expression
\begin{align*}
\label{revn}
\d_i\left(\rev_n (x)\right)=
\begin{cases}
\d_{n+1-i}(x), \  &i=1,2\cdots n,
\\
0, &\mbox{otherwise},
\end{cases}
\end{align*}
in other words, setting $\d_i(x) = x_i$, we have $\rev_n (0.x_1 x_2
... x_n) = 0.x_n ... x_2 x_1$.
\end{definition}

%The role of the index $ n$ is suppressed in this notation.
\begin{definition}
Let $x,\sigma\in [0,1)$ where $\sigma$ has $n$ binary digits. We define the number $x\oplus \sigma$ as
\label{d.oplus}
$$\d_i(x\oplus \sigma)=\d_i(x)+\d_i(\sigma) \space \textup{ mod 2},$$
i.e. the $i^{th}$ digit of $x$ changes if $\d_i(\sigma)=1$ and stays
the same if $\d_i(\sigma)=0$. In the literature this operation is
called \emph{digit scrambling} or \emph{digital shift}.
\end{definition}%%%%%%%%%%%%%%%%%%%%%%%%%%%%%%  DEFINITION DEFINITION DEFINITION
\begin{remark} %%%%%%%%%%%%%%%%%%%%%%%%%%%%%%  REMARK REMARK REMARK REMARK
We stress at this point that when we define a digit scrambling we only use the first $n$ binary digits of the number $\sigma\in[0,1)$. As a result, for each given positive integer $n$ there are exactly $2^n$ such digital shifts, that is, the number of digital shifts is finite. The choice of a real number $\sigma\in[0,1)$ to represent this operation is just a matter of notational convenience.
\end{remark} %%%%%%%%%%%%%%%%%%%%%%%%%%%%%%  REMARK REMARK REMARK REMARK

We are now ready to define the digit-scrambled van der Corput set.
\begin{definition}
For an integer $ n\ge 1$ and a number $\sigma\in[0,1)$ we define the $\sigma$-digit scrambled
\emph{van der Corput} set $\mathcal{V}_{n,\sigma}$ as
\label{d.vdc}
\begin{equation*}
 \mathcal{V}_{n,\sigma} = \{ v_{n,\sigma}(\tau):\, \ \tau=0,1,\ldots,2^n-1\},
\end{equation*}
where
\begin{equation*}
v_{n,\sigma}(\tau)=\bigg(\frac{\tau} {2^n}, \operatorname {rev_n}
\bigg(\frac{\tau}{2^n}\oplus \sigma\bigg)\bigg)+ (2 ^{-n-1}, 2
^{-n-1}).
\end{equation*}
It is clear that the digit-scrambled van der Corput set has
cardinality $|\mathcal{V}_{n,\sigma}|=2^n.$ We should notice that
the roles of $x$ and $y$ coordinates are symmetric, since we can
write $\mathcal{V}_{n,\sigma} = \{(\rev_n(\tau/2^n \oplus \sigma')
,{\tau}/ {2^n})+ (2 ^{-n-1}, 2 ^{-n-1}):\, \
\tau=0,1,\ldots,2^n-1\}$ with $\sigma'=\rev_n (\sigma)$.
\end{definition}

With the notation introduced above, the standard van der Corput set
$$\mathcal{V}_n = \{ (0.x_1 x_2 ... x_n 1 , 0.x_n ... x_2 x_1 1):\,
x_i =0,1 \} $$  is just $\mathcal{V}_n=\mathcal{V}_{n,0}.$ Note that
our definition differs from the classical by
 the shift $ (2 ^{-n-1}, 2 ^{-n-1})$.  This shift  `pads' the binary expansion of the elements by a final $1$ in the $(n + 1)^\textup{st}$
 place, and ensures that the average value of each coordinate is $ \tfrac 12 $:
\begin{equation}
\label{e.centralized}
 2 ^{-n}\sum _{ (x,y)\in \mathcal V_{n,\sigma}} x = 2 ^{-n}\sum _{ (x,y)\in \mathcal V_{n,\sigma}} y = \frac{1}{2}.
\end{equation}
This is just a technical modification that will simplify our formulas and calculations.

%%%%%%%%%%%%%%%%%%%%%%%%%%%%%% PROPOSITION PROPOSITION PROPOSITION

The following proposition describes which points of the van der Corput set $\mathcal{V}_{n,\sigma}$ fall into any given dyadic rectangle.
\begin{proposition}
\label{p.net} Let $k,l\in\mathbb{N}$ and $i\in\{0,1,\ldots,2^k-1\}$,
$j\in\{0,1,\ldots,2^l-1\}. $ Consider a dyadic rectangle
\begin{equation*}
R=\left[\frac{i}{2^k},\frac{i+1}{2^k}\right)\times\left[\frac{j}{2^l},\frac{j+1}{2^l}\right).
\end{equation*}
Then the set $\mathcal{V}_{n,\sigma}\cap R$ consists of the points $v_{n,\sigma}(\tau)$ where
\begin{align*}
\d_m\big(\frac{\tau}{2^n}\big)=
\begin{cases}
\d_m(\frac{i}{2^k}),\ &m=1,2\cdots, k,
\\
\d_{n+1-m}(\frac{j}{2^k}) + \d_m(\sigma) \mod 2, \  &m=n+1-l,
\cdots, n.
\end{cases}
\end{align*}
\end{proposition}
%%%%%%%%%%%%%%%%%%%%%%%%%%%%%% PROPOSITION PROPOSITION PROPOSITION

%%%%%%%%%%%%%%%%%%%%%%%%%%%%%% PROOF PROOF PROOF
\begin{proof}
Let $(x,y)$ be any point $[0,1)^2$. It is easy to see that $(x,y)\in R$ if and only if
\begin{align*}
\d_q(x)&=\d_q\big(\frac{i}{2^k}\big)\  \ \ \textup{for all} \ \ \ q=1,2,\ldots, k, \ \ \textup{ and} \\
\d_r(y)&=\d_r\big(\frac{j}{2^l}\big)\ \ \ \ \textup{for all} \  \ \  r=1,2,\ldots, l.
\end{align*}
The proposition is now a simple consequence of the structure of the
van der Corput set.
\end{proof}
%%%%%%%%%%%%%%%%%%%%%%%%%%%%%% PROOF PROOF PROOF
Some remarks are in order:
\begin{remarks}\label{r.npoints}

\item{ When $k+l<n$ there are exactly $2^{n-(k+l)}$ points of the van der Corput set inside the canonical rectangle $R$.
Indeed, the conditions of Proposition \ref{p.net} only specify the first $k$ and last $l$ binary digits of the $x-$coordinates of the points $v_{n,\sigma}(\tau)$. }

\item {When $k+l>n$ it might happen that the set of conditions in proposition \ref{p.net} is void (observe that the system is overdetermined in this case). }

\item {Finally, when $k+l=n$, that is when the rectangle $R$ has volume $|R|=2^{-n}$,
the system of equations in \ref{p.net} gives a unique point of the
van der Corput set inside $R$. So, for fixed $n$, the van der Corput
set $\mathcal{V}_{n,\sigma}$ is a \emph{net}: every dyadic
 rectangle of volume $N^{-1}=2^{-n}$ contains exactly one point. This has the well-known consequence, see \cite{MR1697825}, that
\begin{equation}\label{e.schmidtsharp}
\norm D_N(\mathcal{V}_{n,\sigma}). \infty. \lesssim \log N.
\end{equation}}
This fact is independent of the digit scrambling $\sigma$ and holds
in particular for the standard van der Corput set $\mathcal{V}_n$
(\cite{Cor35}, \cite{MR0066435}). In view of Schmidt's Theorem
\eqref{e.schmidt} this means that the van der Corput set is
\emph{extremal} in terms of measuring the Discrepancy function in
$L^\infty$. However, the same is not true if one is interested in
meeting the lower bound in Roth's Theorem, that is, the standard van
der Corput set $\mathcal{V}_n$ is not extremal in terms of measuring
the Discrepancy function in $L^2$. The lemma below explains this
fact. In particular it shows that the $L^2$ discrepancy of
$\mathcal{V}_n$ is big because of a single `zero-order' Haar
coefficient, i.\thinspace e.\thinspace the mean $\int D_N$. The lemma also shows that digit
scrambling provides a remedy for this shortcoming. This fact has been observed by Chen in \cite{MR711520} where the author uses digit scrambling in order to obtain the best possible $L^p$ upper bounds for a general class of 'one point in a box' sets in general dimension (see the case $k+l=n$ in the remarks above). We also note that
similar calculations, albeit slightly less general, have been carried out in \cite{HAZ}. We include a proof of this Lemma for the
sake of completeness.
\end{remarks}

%%%%%%%%%%%%%%%%%%%%%%%%%%%%%% LEMMA LEMMA LEMMA
\begin{lemma}\label{l.integral}We have
$$\int_0^1\!\!\int_0^1D_N(\mathcal{V}_{n,\sigma})\ dxdy=\frac{1}{4}\left(\frac{n}{2}-\sum_{k=1}^n \d_k(\sigma)\right) .$$
In particular
$$\int_0^1\!\!\int_0^1D_N(\mathcal{V}_{n})\ dxdy=\frac{n}{8}.$$
On the other hand, if $\sum_{k=1}^n \d_k(\sigma)=n/2$, i.e. half of
the digits are scrambled, then
$$\int_0^1\!\!\int_0^1D_N(\mathcal{V}_{n,\sigma})\ dxdy=0.$$
\end{lemma}
%%%%%%%%%%%%%%%%%%%%%%%%%%%%%% LEMMA LEMMA LEMMA

%%%%%%%%%%%%%%%%%%%%%%%%%%%%%% PROOF PROOF PROOF
%%%%%%%%%%%%%%%%%%%%%%%%%%%%%% PROOF PROOF PROOF
\begin{proof}

As usually, we write $N=2^n$. We have
\begin{align*}
I&\coloneqq\int_0^1\!\!\int_0^1 D_N(\mathcal{V}_{n,\sigma})(x,y)\ dxdy=-N/4+\sum_{\tau=0}^{N-1}\int_0 ^1\int_0 ^1 \mathbf 1_{[0,x]\times[0,y]}\left(v_{n,\sigma}\left(\tau/N\right)\right)dxdy\\&= -N/4+\sum_{\tau=0}^{N-1}\left(1-\frac{\tau}{N}-\frac{1}{2N}\right)\left(1-\rev_n\left(\frac{\tau}{N}\oplus\sigma\right)-\frac{1}{2N}\right).
\end{align*}
Using $\eqref{e.centralized}$ we get
\begin{equation}\label{e.i}
I=-\frac{N}{4}+\frac{1}{2}-\frac{1}{4N}+\sum_{\tau=0}^{N-1}\frac{\tau}{N}\cdot
\rev_n\left(\frac{\tau}{N}\oplus\sigma\right).
\end{equation}
Now expand the sum above using the binary representation of the
summands as follows:
\begin{align}\label{e.diag} \sum_{\tau=0}^{N-1}\frac{\tau}{N}\cdot \rev_n\left(\frac{\tau}{N}\oplus\sigma\right)&=\sum_{\tau=0}^{N-1}\sum_{k=1}^n\sum_{l=1}^n\frac{\d_k\left(\frac{\tau}{N}\right) \d_{l}\left(\rev_n\left(\frac{\tau}{N}\oplus\sigma\right)\right)}{2^{k+l}} \notag
\\&=\sum_{\tau=0}^{N-1}\sum_{k=1}^n\sum_{l=1}^n\frac{\d_k\left(\frac{\tau}{N}\right) \d_{n+1-l}\left(\frac{\tau}{N}\oplus\sigma\right)}{2^{k+l}}
\\&=\sum_{k=1}^n\sum_{l=1}^n \frac{1}{2^{k+l}} \sum_{\tau=0}^{N-1} {\d_k\left(\frac{\tau}{N}\right) \d_{n+1-l}\left(\frac{\tau}{N}\oplus\sigma\right)}.\notag
\end{align}
Finally observe that if $s,t\in\{1,2,\ldots,n\}$ then
\begin{align}\label{e.orthodigit}
\sum_{\tau=0}^{N-1}\d_s\left(\frac{\tau}{N}\right)\d_t\left(\frac{\tau}{N}\oplus\sigma\right)=
\begin{cases}
\frac{N}{2}
\left(1-\d_s(\sigma)\right), &s=t
\\
\frac{N}{4}
 , &s\neq t.
\end{cases}
\end{align}
Indeed, when $s=t$, the terms in the sum above are non-zero exactly
when  $\d_s(\frac{\tau}{N})=1$ and $\d_s(\sigma)=0$, and hence the
first equality. The case $s\neq t $ is similar.

Using \eqref{e.orthodigit} and \eqref{e.diag} we get
\begin{align*}
\sum_{\tau=0}^{N-1} \frac{\tau}{N} \, \rev_n \left(
\frac{\tau}{N}\oplus\sigma \right) &=
\frac{n}{8}-\frac{1}{4}\sum_{k=1}^n
\d_k(\sigma)+\frac{N}{4}-\frac{1}{2}+\frac{1}{4N},
\end{align*}
which, combined with \eqref{e.i}, completes the proof.
\end{proof}
%%%%%%%%%%%%%%%%%%%%%%%%%%%%%% PROOF PROOF PROOF
%%%%%%%%%%%%%%%%%%%%%%%%%%%%%% PROOF PROOF PROOF
{\emph {Remark.}} We should point out that in \cite{KrPil} it has
been shown that the $L^2$ norm of the Discrepancy of the
digit-scrambled van der Corput set depends only on the number of
$1$'s in $\sigma$, and not their distribution.

%%%%%%%%%%%%%%%%%%%%%%%%%%%%%% SECTION  SECTION SECTION
%%%%%%%%%%%%%%%%%%%%%%%%%%%%%% SECTION  SECTION SECTION
\section{Haar Coefficients for the Digit-Scrambled van der Corput Set} %\label{s.}

In this section we will work with the digit-scrambled van der Corput set $\mathcal{V}_{n,\sigma}$ as defined in Section \ref{s.vandercorput}, where $\sigma\in[0,1)$ is arbitrary and $N=2^n$. We will just write $D_N$ for the discrepancy function of $\mathcal{V}_{n,\sigma}$. The following Lemma records the main estimate for the Haar 
coefficients of $D_N$ and is the core of the proof for the upper bounds in Theorems \ref{t.vdc} and \ref{t.bmo}.

%%%%%%%%%%%%%%%%%%%%%%%%%%%%%% LEMMA LEMMA LEMMA
\begin{lemma}\label{l.haarcoeffs}  For any dyadic rectangle
 $ R \in \mathcal {D}^2$ we have
\begin{equation*}
\lvert  \ip D_N, h_R, \rvert \lesssim \frac{1}{ N}.
\end{equation*}
\end{lemma}
%%%%%%%%%%%%%%%%%%%%%%%%%%%%%% LEMMA LEMMA LEMMA
We need to consider dyadic rectangles of the form
%\begin{equation*}
%\label{eq.ofR}
$R=\left[\frac{i}{2^k};\frac{i+1}{2^k}\right)\times\left[\frac{j}{2^l};\frac{j+1}{2^l}\right)$,
%\end{equation*}
where $k,l\in\mathbb{N}$ and $i\in\{0,1,\ldots,2^k-1\}$,
$j\in\{0,1,\ldots,2^l-1\}. $ The proof will be divided in two cases,
depending on whether the volume of $R$ is `big' or `small'. %In other
%words, we consider different cases for the number $k+l$ which gives
%us some information about the number of points of the van der Corput
%set contained in $R$.

We will use an auxiliary function  to help us write down formulas for the inner product of the counting part with the Haar function corresponding to the rectangle $R$.  In particular, $\phi:\mathbb R\rightarrow \mathbb R$ is the periodic function
\begin{align*}
\phi (x) =
\begin{cases}
\{x\}, \ &0<\{x\}<\tfrac 12
\\
1-\{x\}, \  &\tfrac 12 <\{x\} < 1 ,
\end{cases}
\end{align*}
where $\{x\}$ is the fractional part of $x$.
Observe that the function $\phi$ is the periodic extension of the anti-derivative of the Haar function on [0,1]. See Figure \ref{f.phi}.

Let $p=(p_x,p_y)\in[0,1)^2$. A moment's reflection  allows us to
write

\begin{align}\label{e.phi}
\ip \mathbf 1_{[\vec p,\vec 1)} ,h_R,=
\begin{cases}
|R|    \phi (2^kp_x)\phi (2^l p_y), \ &p\in R,
\\
0, &\textup{otherwise}.
\end{cases}
\end{align}

We also record two simple properties of the function $\phi$ that will be useful in what follows. First, for $x\in\mathbb{R}$,
\begin{equation}\label{e.plushalf}
\phi (x)+\phi \bigg(x\oplus \frac{1}{2}\bigg) = \frac{1}{2}.
\end{equation}
Second, $\phi$ is a `Lipschitz' function with constant $1$. For
$x,y\in \mathbb{R}$,
\begin{equation}\label{e.lips}
\left|\phi(y)-\phi(x)\right|\leq |\{y\}-\{x\}|.
\end{equation}

%%%%%%%%%%%%%%%%%%%%%%%%%%%%%% PROOF PROOF PROOF
%%%%%%%%%%%%%%%%%%%%%%%%%%%%%% PROOF PROOF PROOF
\begin{proof}[Proof of Lemma \ref{l.haarcoeffs} when $|R|<\frac{4}{N}$]
We fix a dyadic rectangle $R$ with $|R|<\frac{4}{N}$. We treat the linear part and the counting part separately.

For the linear part we have that
\begin{equation*}
\ip L_N,h_R , = \frac{N|R|^2}{4^2}\lesssim \frac{1}{N}.
\end{equation*}

Now notice that since $k+l>n-2$, there are at most $2$ points in
$\mathcal{V}_{n,\sigma}\cap R$. Since $\phi$ is obviously bounded by
$1$, formula \eqref{e.phi} implies
\begin{equation*}
|\ip C_{\mathcal{V}_{n,\sigma}},h_R,|\leq |R| \sum_{p\in\mathcal{V}_{n,\sigma}\cap R}  \phi (2^kp_x)\phi (2^l p_y) \leq 4|R|\lesssim\frac{1}{N}.
\end{equation*}
Summing up the estimates for the linear and the counting part completes the proof.
\end{proof}
%%%%%%%%%%%%%%%%%%%%%%%%%%%%%% PROOF PROOF PROOF
%%%%%%%%%%%%%%%%%%%%%%%%%%%%%% PROOF PROOF PROOF

%%%%%%%%%%%%%%% Figure
\begin{figure}
\begin{tikzpicture}

\draw[->](0,-.5) -- (0,1.5); \draw (.125,1) -- (-.125,1) node[left]
{$ 1$}; \draw[->] (-.5,0) -- (3.5,0); \draw (1.5,.125) --
(1.5,-.125) node[below] {$ 1/2$}; \draw (0,0) -- (1.5,1) -- (3,0);
\end{tikzpicture}

\caption{The graph of the  function $ \phi $.} \label{f.phi}
\end{figure}
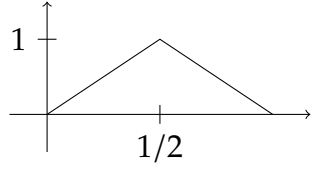
%%%%%%%%%%%%%%% Figure

%%%%%%%%%%%%%%%%%%%%%%%%%%%%%% PROOF PROOF PROOF
%%%%%%%%%%%%%%%%%%%%%%%%%%%%%% PROOF PROOF PROOF
\begin{proof}[Proof of Lemma \ref{l.haarcoeffs} when $|R|\geq \frac{4}{N}$]
The proof of the case $|R|\geq \frac{4}{N}$ is much more involved as
this is the typical case where the rectangle contains `many' points
of the point set $\mathcal{V}_{n,\sigma}$. Before going into the
details of the proof we will discuss the structure of the set $R\cap
\mathcal{V}_{n,\sigma}$ in order to organize and simplify the
calculations that follow.

First, notice that the condition $|R|\geq \frac{4}{N}$ implies that
$n-(k+l)\geq 2$. In other words, there are at least $4$ points in
the set $R\cap \mathcal{V}_{n,\sigma}$ according to Proposition
\ref{p.net} and Remark \ref{r.npoints}. To be more precise, let us
look at a point $p=(x,y)\in \mathcal{V}_{n,\sigma}$. The $x$-coordinate
can be written in the form $x=0.x_1x_2\ldots x_n1,$ where
$x_i=\d_i(x)$, for $i=1,2,\ldots,n$. %Remember that the
%$x$-coordinates of the van der Corput set are of the form
%$\frac{\tau}{2^n}+2^{-n-1}$ for some $\tau\in\{1,2,\ldots,2^n-1\}$
%hence their binary expansion is of length $n+1$, the
%$n+1^\textup{st}$ digit being always $1$.
The first $k$  and the last $l$ binary digits of $x$ are determined
by the fact that $x\in R$ (Proposition \ref{p.net}). That leaves us
with at least $2$ `free' digits for $x$
$$x=0.x_1\ldots x_k,*,\ldots,*,x_{n-l+1} \ldots x_n 1.$$
We group all points in $\mathcal{V}_{n,\sigma}\cap R$ in quadruples
according to the choices for the first and last `free` digits
$x_{k+1}$ and $x_{n-l}$. In particular, we consider
\emph{quadruples} \eqref{e.quad} of points in
$\mathcal{V}_{n,\sigma}\cap R$ with $x$-coordinates of the form:

\[ \begin{array}{ccc}\tag{Q}\label{e.quad}
&0.x_1\ldots x_k \ 0 \ x_{k+2} \ldots,x_{n-l-1} \ 0\ x_{n-l+1}\ldots x_n 1,\\
&0.x_1\ldots x_k \ 0 \ x_{k+2} \ldots,x_{n-l-1} \ 1 \ x_{n-l+1}\ldots x_n 1,\\
&0.x_1\ldots x_k \ 1 \ x_{k+2} \ldots,x_{n-l-1} \ 0 \ x_{n-l+1}\ldots x_n 1,\\
&0.x_1\ldots x_k  \ 1\  x_{k+2} \ldots,x_{n-l-1} \ 1\ x_{n-l+1}\ldots x_n 1.
\end{array}\]

\begin{figure}
\begin{tikzpicture}
\draw (8.5,1.5) node {$R$}; \draw (0,0) rectangle (8,2);
\draw[fill,fill opacity=0.5,lightgray] (0,0) rectangle (4,1);
\draw[fill,fill opacity=0.5,lightgray] (4,1) rectangle (8,2);
\foreach \position in { (.675,.125), (4.675,.325), (4.775,1.325),
(.775,1.125)} { \draw[fill=black] \position circle (2pt);}%
\draw  (.65,.5) node  {$ (u,v)$};
%\draw (.675,.125) node (end) {}; %%
%\draw  (-2,.75) node (beg) {$ (u,v)$}; %%
%\path[->] (beg) edge [out=-90, in=140] (end);
\end{tikzpicture}
\caption{The quadruple $Q$. } \label{f.four}
\end{figure}
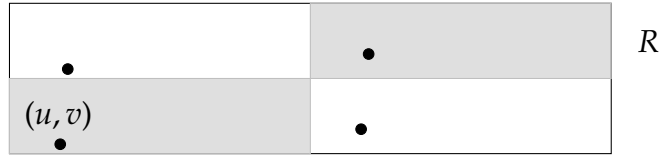

There are exactly $2^{n-(k+l)-2}=\frac{N|R|}{4}$ such quadruples. Let's index the quadruples \ref{e.quad} arbitrarily as $Q_r$, $r=1,2,\ldots,\frac{N|R|}{4}$. Observe that we can write
\begin{align}\label{e.qsplit}
\ip D_N, h_R, = \sum_{p\in\mathcal{V}_{n,\sigma}\cap R} \ip \textbf{1}_ {[\vec p, \vec 1)} ,h_R,  -\frac{N|R|^2}{16}=
\sum_{r=1} ^{\frac{N|R|}{4} } \bigg(\sum _{p\in Q_r} \ip \textbf{1}_ {[\vec p, \vec 1)} ,h_R,  -\frac{|R|}{4}\bigg).
\end{align}
The following Proposition exploits  large cancellation within these
quadruples.

\begin{proposition}\label{p.cancel}
\begin{equation*}
\ABs { \sum _{p\in Q_r} \ip \mathbf 1_ {[\vec p, \vec 1)} ,h_R,  -\frac{|R|}{4} }  \lesssim \frac{1}{N^2|R|}.
\end{equation*}
\end{proposition}
Let assume Proposition \ref{p.cancel} for a moment in order to
complete the proof of Lemma \ref{l.haarcoeffs}. Indeed, Proposition
\ref{p.cancel} together with equation \eqref{e.qsplit} immediately
yield
\begin{align*}
\ip D_N, h_R, \lesssim \sum_{r=1} ^{\frac{N|R|}{4} }\frac{1}{N^2|R|}
\lesssim \frac{1}{N}.
\end{align*}
This completes the proof modulo Proposition \ref{p.cancel}.
\end{proof}

\begin{proof}[Proof of Proposition \ref{p.cancel}] For the proof of the proposition we will fix a
$Q=Q_r$ and suppress the index $r$ since it does not play any role.
Suppose $p=(u,v)$ is any of the points with $x$-coordinate as in
\eqref{e.quad} and $y$-coordinate $v$ such that
$p\in\mathcal{V}_{n,\sigma}$ . Then it is easy to see that the
quadruple \eqref{e.quad} consists of the four points which can be
written in the form:
\begin{align*}\tag{Q}
\begin{cases}
&(u,v),\\
&(u\oplus 2^{-k-1},v\oplus2^{-n+k}), \\
&(u\oplus 2 ^{-n+l},v\oplus2^{-l-1}),\\
&(u\oplus 2^{-n+l}\oplus 2^{-k-1},v\oplus2^{-n+k}\oplus2^{-l-1}).
\end{cases}
\end{align*}
See also Figure \ref{f.four}.

We invoke equation \eqref{e.phi} to write
\begin{align}\label{e.phiprop}
 \sum _{p\in Q} \ip \textbf{1}_ {[\vec p, \vec 1)} ,h_R,  -\frac{|R|}{4} = |R|\Big( \sum _{p\in Q} \phi(2^kp_x)\phi(2^lp_y)  -\frac{1}{4} \Big)\eqqcolon |R|\big(\Sigma-\frac{1}{4}\big).
\end{align}
We have
\begin{align*}
 \Sigma  &= \phi(2^ku)\phi(2^lv)\\
 &+\phi(2^k u\oplus \frac{1}{2})\phi(2^l ( v\oplus2^{-n+k}))\\
 &+\phi(2^k ( u\oplus 2 ^{-n+l}))\phi(2^l v\oplus \frac{1}{2})\\
 &+\phi(2^k u\oplus 2^k\cdot 2^{-n+l}\oplus \frac{1}{2})\phi(2^l v\oplus2^l\cdot 2^{-n+k}\oplus\frac{1}{2}).
 \end{align*}
Using equation \eqref{e.plushalf} we get
\begin{equation*}
\Sigma=\frac{1}{4}+\big[\phi(2^ku)- \phi(2^k ( u\oplus 2 ^{-n+l}))\big]\big[\phi(2^lv)-\phi(2^l ( v\oplus2^{-n+k}))\big].
\end{equation*}
Finally, using the fact the the function $\phi$ is Lipschitz
\eqref{e.lips} we have
\begin{align*}
\ABs{\Sigma -\frac{1}{4}} \leq (2^{-n+l+k})^2=\frac{1}{N^2 |R|^2}.
\end{align*}
This estimate together with equation \eqref{e.phiprop} completes the proof.
\end{proof}
%%%%%%%%%%%%%%%%%%%%%%%%%%%%%% PROOF PROOF PROOF
%%%%%%%%%%%%%%%%%%%%%%%%%%%%%% PROOF PROOF PROOF

Lemma \ref{l.haarcoeffs} has an analogue in the case of Haar
functions $h^{1,0}_{[0,1]\times I}$ and $h^{0,1} _{I\times [0,1]}$,
where $I\in \mathcal{D}$. Observe also that the inner product that
corresponds to $h^{1,1} _{[0,1]^2}$ is the content of Lemma
\ref{l.integral} of the previous section.

%%%%%%%%%%%%%%%%%%%%%%%%%%%%%% LEMMA LEMMA LEMMA
%%%%%%%%%%%%%%%%%%%%%%%%%%%%%% LEMMA LEMMA LEMMA
\begin{lemma}\label{l.semihaarcoeffs}
For $I\in\mathcal{D}$ we have the estimates
\begin{align*}
&\abs{\ip D_N, h^{0,1}_{I\times [0,1]}, }\lesssim |I|,\\
&\abs{\ip D_N, h^{1,0}_{[0,1]\times I}, }\lesssim |I|.
\end{align*}
\end{lemma}
%%%%%%%%%%%%%%%%%%%%%%%%%%%%%% LEMMA LEMMA LEMMA
%%%%%%%%%%%%%%%%%%%%%%%%%%%%%% LEMMA LEMMA LEMMA

%%%%%%%%%%%%%%%%%%%%%%%%%%%%%% PROOF PROOF PROOF
%%%%%%%%%%%%%%%%%%%%%%%%%%%%%% PROOF PROOF PROOF

\begin{proof}
It suffices to prove just the first estimate in the
statement of the Lemma. %To see that note that roles of
%$x=\frac{\tau}{2^n}$ and
%$y=\rev_n\big({\frac{\tau}{2^n}\oplus\sigma} \big)$ in the
%definition of the van der Corput set can be interchanged.
The proof proceeds in a more or less analogous fashion as the proof
of Lemma \ref{l.haarcoeffs}. We fix a dyadic interval
$I=\left[\frac{i}{2^k},\frac{i+1}{2^k} \right)$ and write
$h_I=h^{0,1} _{I\times[0,1]}$. We need an analogue of formula
\eqref{e.phi} which in this case becomes

\begin{align}\label{e.phix}
\ip \mathbf 1_{[\vec p,\vec 1)} ,h_I,=
\begin{cases}
|I|    \phi (2^kp_x)(1-p_y), \ &p_x\in I,
\\
0, &\textup{otherwise}.
\end{cases}
\end{align}

As in the proof of Lemma \ref{l.haarcoeffs}, we need to consider
separately the case of small volume and large volume rectangles. The
small volume case here is $|I|\leq \frac{2}{N}$. Note that in this
case there are at most $2^{n-k}\leq 2$ points of the van der Corput
set whose $x$ coordinate lies in $I$. Using equation \eqref{e.phix}
we trivially get the desired estimate as in the proof of the
corresponding case of Lemma \ref{l.haarcoeffs}.

We now turn to the main part of the proof, namely the estimate
\begin{equation*}
\abs{\ip D_N, h^{1,0}_{I\times [0,1]}, }\lesssim |I|,
\end{equation*}
when $|I|>\frac{2}{N}$. Instead of the quadruples \eqref{e.quad}, we
now group the points of the van der Corput set with $x$-coordinate
in $I$, into \emph{pairs} \eqref{e.pair} of the form:

\[ \begin{array}{ccc}\tag{P}\label{e.pair}
&0.x_1\ldots x_k \ 0\  x_{k+2} \ldots x_n 1,\\
&0.x_1\ldots x_k \ 1 \ x_{k+2} \ldots x_n 1.
\end{array}\]

If $(u,v)$ is one of the two points in \eqref{e.pair}, we also have the description:
\begin{align*}\tag{P}
\begin{cases}
&(u,v),\\
&(u\oplus 2^{-k-1},v\oplus2^{-n+k}).
\end{cases}
\end{align*}

There are $2^{n-k-1}$ such pairs and let's index them arbitrarily as
$P_r$, $r=1,2,\ldots ,2^{n-k-1}$. We write
\begin{align}
\ip D_N, h_I , = \sum_{p\in\mathcal {V}_{n,\sigma}\cap I \times[0,1] } \ip \textbf{1}_ {[\vec p, \vec 1)} ,h_I,  -\frac{N|I|^2}{8}&=
\sum_{r=1} ^{2^{n-k-1}} \sum _{p\in P_r} \ip \textbf{1}_ {[\vec p, \vec 1)} ,h_I,  -\frac{N|I|^2}{8}.
\end{align}
Now for any pair \eqref{e.pair} we use \eqref{e.phix} to write
\begin{align*}
\sum _{p\in P} \ip \textbf{1}_ {[\vec p, \vec 1)} ,h_I, &= |I|\ \phi(2^ku)(1-v)+|I|\ \phi(2^k(u\oplus2^{-k-1}))(1-v\oplus 2^{-n+k})\\
&=|I|\ \big[ \phi(2^ku)+\phi(2^ku\oplus 2^{-1})\big] \ (1-v)
\\&+|I|\  \phi(2^ku\oplus2^{-1})\ (v-v\oplus2^{-n+k})
\\&=\frac{1}{2}|I|(1-v) + |I|\  \phi(2^ku\oplus2^{-1})\ (v-v\oplus2^{-n+k}).
\end{align*}
where in the last equality we have used \eqref{e.plushalf}. Using
the fact that $\abs{v-v\oplus2^{-n+k}} = 2^{-n+k}$ and assuming
$\d_{n-k} (v)=0$, it is routine to check that
\begin{equation}\label{e.almosthere}
\ip D_N, h_I ,=|I|\ \bigg\{\frac{1}{2} \sum_{r=1} ^{2^{n-k-1}} (1-v_r)-2^{n-k-3}+\mathcal{O}(1)\bigg\},
\end{equation}
where $v_r$ are $y$-coordinates of the form
\begin{align*}
v_r&=0.Y_1\ldots Y_{n-k-1} 0  y_{n-k+1}\ldots y_n1%\\&=0.Y_1\ldots
%Y_{n-k-1}+2^{-n+k}0.y_{n-k}\ldots y_n1.
\end{align*}
%The digit $y_{n-k}$ depends on the wheteher our original point
%$(u,v)$ was in the first or the second half of the interval $I$,
%however it is fixed.
The digits $y_{n-k+1}$ up to $y_n$ are fixed
because of the digit reversal structure of the van der Corput set.
%The point is that the first $n-k-1$ digits of $v_r$ can be
%considered as independent random variables taking values $\{0,1\}$
%with probability $\frac{1}{2}$.
We can then estimate the sum in the previous expression as follows:
\begin{align*}
 \sum_{r=1} ^{2^{n-k-1}} (1-v_r)&= 2^{n-k-1}-\frac{1}{2}2^{n-k-1}\bigg(1 -2^{-n+k+1}\bigg) + \mathcal{O}(1)=2^{n-k-2}+\mathcal{O}(1).
\end{align*}

Substituting in \eqref{e.almosthere} we get
\begin{equation*}
\ip D_N, h_I ,=|I|\ \bigg\{\frac{1}{2}\big(2^{n-k-2}+\mathcal{O}(1)\big) -2^{n-k-3} +\mathcal{O}(1)\bigg\}\lesssim |I|,
\end{equation*}
which completes the proof.
\end{proof}

%%%%%%%%%%%%%%%%%%%%%%%%%%%%%% PROOF PROOF PROOF
%%%%%%%%%%%%%%%%%%%%%%%%%%%%%% PROOF PROOF PROOF

%%%%%%%%%%%%%%%%%%%%%%%%%%%%%% SECTION  SECTION SECTION
%%%%%%%%%%%%%%%%%%%%%%%%%%%%%% SECTION  SECTION SECTION
\section{BMO Estimates for the Discrepancy Function} %\label{s.}

This section is devoted to the proofs of Theorems \ref{t.bmolower}
and \ref{t.bmo}. We recall that the Dyadic Chang-Fefferman $
\operatorname {BMO}_{1,2}$ is defined to consist of those square
integrable functions $ f$ in the linear span of $ \{h_R \mid R\in
\mathcal D ^2 \}$, for which we have
\begin{equation*}
\norm f. \operatorname {BMO}_{1,2} . \coloneqq \sup _{ U\subset
[0,1] ^2 } \Biggl[ \lvert  U\rvert ^{-1} \sum _{\substack{R \in
\mathcal D ^2
\\ R\subset U}}
\frac {\ip f ,h _R , ^2  } {\lvert  R\rvert }
 \Biggr] ^{1/2} < \infty \,.
\end{equation*}
%The supremum is taken over all measurable subsets $ U\subset [0,1]
%^2 $.
We begin with the proof of Theorem \ref{t.bmolower} which is
essentially just a repetition of the argument used in Proposition
\ref{p.rvec}.

%%%%%%%%%%%%%%%%%%%%%%%%%%%%%% PROOF PROOF PROOF
%%%%%%%%%%%%%%%%%%%%%%%%%%%%%% PROOF PROOF PROOF
\begin{proof}[Proof of Theorem \ref{t.bmolower}] We fix a distribution $\mathcal{A}_N$ of $N$ points in the unit square
 and take $n$ such that $2N<2^n\leq 4N$. For the special choice of $U=[0,1]^2$ we have
\begin{equation*}
\norm D_N. \operatorname {BMO}_{1,2} .^2 \geq \sum_{\vec r \in\mathbb{H}_n} \sum_{\substack{ R\in\mathcal{R}_{\vec r}  \\ R\cap \mathcal{A}_N= \emptyset }} \frac{\ip D_N,h_R,^2}{|R|}.
\end{equation*}
Consider a rectangle $R\in \mathcal{R}_{\vec r}$ which does not
contain any points of $\mathcal{A}_N$. Then
\begin{equation*}
\ip D_N,h_R , = - \ip L_N, h_R, = - \frac{|R|^2}{4^2}.
\end{equation*}
As a result,
\begin{equation*}
\norm D_N. \operatorname {BMO}_{1,2} .^2 \gtrsim \sum_{\vec r \in\mathbb{H}_n}\sum_{\substack{ R\in\mathcal{R}_{\vec r}  \\ R\cap \mathcal{A}_N= \emptyset }} N^2 |R|^3\gtrsim \frac{1}{N} \sum_{\vec r \in\mathbb{H}_n}\sharp \{R\in\mathcal R_{\vec r}, \ R\cap \mathcal{A}_N=\emptyset\}.
\end{equation*}
For fixed $\vec r \in \mathbb{H}_n$ we have $\sharp \{R\in\mathcal R_{\vec r}, R\cap \mathcal{A}_N=\emptyset\}\geq N$, arguing as in the proof of Proposition \ref{p.rvec}. Thus we get
\begin{equation*}
\norm D_N. \operatorname {BMO}_{1,2} .^2 \gtrsim \sum_{\vec r \in\mathbb{H}_n} 1\gtrsim n.
\end{equation*}
This completes the proof since $n\simeq \log N$.
\end{proof}
%%%%%%%%%%%%%%%%%%%%%%%%%%%%%% PROOF PROOF PROOF
%%%%%%%%%%%%%%%%%%%%%%%%%%%%%% PROOF PROOF PROOF

We proceed with the proof of the upper bound in Theorem \ref{t.bmo}.
Our extremal set of cardinality $N=2^n$ will be
$\mathcal{V}_{n,\sigma}$ for arbitrary $\sigma\in[0,1)$, as defined
in Definition~\ref{d.vdc}. We will just write $D_N$ for the
Discrepancy function of the digit-scrambled van der Corput set.

%%%%%%%%%%%%%%%%%%%%%%%%%%%%%% PROOF PROOF PROOF
%%%%%%%%%%%%%%%%%%%%%%%%%%%%%% PROOF PROOF PROOF

\begin{proof}[Proof of Theorem \ref{t.bmo}]
We fix a measurable set $ U\subset [0,1] ^2 $ and consider only
rectangles $R$ in the family $ \{R \in \mathcal D ^2,R\subset U\}$.
We will sometimes suppress the fact that our rectangles are
contained in $U$ to simplify the notation.

The are two estimates that are relevant here, one for large
rectangles and one for small volume rectangles. For the large volume
case, $|R|\geq 2^{-n}$, we have
\begin{align*}
\lvert  U\rvert ^{-1} \sum_{|R|\geq 2^{-n}}
\frac {\ip  D_N ,h _R , ^2  } {\lvert  R\rvert } &=  \lvert  U\rvert ^{-1}\sum_{k=0} ^n  \sum_{\vec r \in \mathbb{H}_k} \sum_{R\in \mathcal{R}_{\vec r}} \frac{ \ip D_N ,h _R , ^2 }{|R|}\\
&\lesssim N^{-2}\lvert  U\rvert ^{-1}\sum_{k=0} ^n 2^k \sum_{\vec r
\in \mathbb{H}_k} \sum_{R\in \mathcal{R}_{\vec r}}1,
\end{align*}
where we have used the estimate $\ip D_N ,h _R , \lesssim
\frac{1}{N}$ of Proposition \ref{l.haarcoeffs}. Now observe that for
fixed $k$ and $\vec r \in\mathbb{H}_k$ there are at most $2^k|U|$
rectangles $R\in\mathcal{R}_{\vec r}$  contained in $U$.
Furthermore, there are $k$ choices for the `geometry' $\vec r \in
\mathbb{H}_k$. We thus get
\begin{align*}
\lvert  U\rvert ^{-1} \sum_{|R|\geq 2^{-n}} \frac {\ip  D_N ,h _R ,
^2  } {\lvert  R\rvert } &\lesssim
 N^{-2}\sum_{k=0} ^n  k(2^{k})^2 \lesssim \frac{n(2^n)^2}{N^2}= n.
\end{align*}

In the small volume term we treat the linear and the counting parts
separately.

For the linear part we use \eqref{e.vv} to get $\ip  L_N ,h _R, =4^{-2}N \lvert R \rvert^2$. So we have
\begin{align*}
\lvert  U\rvert ^{-1} \sum_{|R|< 2^{-n}}
\frac {\ip L_N ,h _R , ^2  } {\lvert  R\rvert }& = \lvert  U\rvert ^{-1}\sum_{k=n+1} ^\infty  \sum_{\vec r \in \mathbb{H}_k} \sum_{R\in \mathcal{R}_{\vec r}} \frac{ \ip L_N ,h _R , ^2 }{|R|}\\
&\simeq N^2\lvert  U\rvert ^{-1}\sum_{k=n+1} ^\infty  \sum_{\vec r \in \mathbb{H}_k} (2^{-k})^3\sum_{R\in \mathcal{R}_{\vec r}} 1.
\end{align*}
Now arguing as in the large volume case we have $\sum_{R\in
\mathcal{R}_{\vec r}} 1\lesssim 2^k |U|$, and thus
\begin{align*}
\lvert  U\rvert ^{-1} \sum_{|R|< 2^{-n}}
\frac {\ip L_N ,h _R , ^2  } {\lvert  R\rvert }& \lesssim N^2 \sum_{k=n+1} ^\infty k(2^{-k})^2 \lesssim n.
\end{align*}

It remains to bound the counting part that corresponds to small
volume rectangles, i.e.
\begin{equation*}
\lvert  U\rvert ^{-1} \sum_{|R|< 2^{-n}}
\frac {\ip C_{\mathcal{V}_{n,\sigma}} ,h _R , ^2  } {\lvert  R\rvert }.
\end{equation*}
Let $ \mathcal R $ be the maximal dyadic rectangles $R$ of area at
most $ 2 ^{-n}$, contained inside $ U$, and such that $h_R$ has
non-zero inner product with the counting part. It is essential to
note that
\begin{equation} \label{e.b1}
\sum _{R\in \mathcal R} \lvert  R\rvert \lesssim  n \lvert  U\rvert\,.
\end{equation}
Indeed, for  each rectangle $ R\in \mathcal R$, the function $ h _R
$ is, as we have observed, orthogonal to each $\mathbf 1_{[\vec p,
\vec 1) } $ with $\vec p$ not in the interior of $ R$. Thus, $ R$
must contain one element of the van der Corput set in its interior.
On the other hand $\mathcal{V}_{n,\sigma}$ is a net so $R$ contains
exactly one point. Now look at all the rectangles in
$R\in\mathcal{R}$, $R=R_x\times R_y$, with a fixed side length
$|R_x|$. The length of this side must be at least $2^{-n}$ in order
for the rectangle to contain a point of the van der Corput set in
its interior, so there are at most $n$ choices for $|R_x|$. On the
other hand, the rectangles in $\mathcal{R}$ with the same side
length must be disjoint since they are maximal and dyadic. Since
they are all contained in $U$, their union has volume at most $U$.
Summing over all possible side lengths $|R_x|$ proves \eqref{e.b1}.

Now, we can write
\begin{align*}
 \lvert  U\rvert ^{-1} \sum _{|R|<2^{-n}}
\frac {\ip  C_{\mathcal{V}_{n,\sigma}}  ,h _R , ^2  } {\lvert  R \rvert }  \leq   \lvert  U\rvert ^{-1} \sum_{R\in \mathcal {R}} \sum_{R^\prime \subseteq R} \frac {\ip  C_{\mathcal{V}_{n,\sigma}}  ,h _{R^\prime }, ^2  } {\lvert  R^\prime \rvert }.
\end{align*}
Note that we have inequality instead of equality, since a rectangle
$R$ can be contained in several maximal rectangles. However, this
does not create any problem.

Let $R\in \mathcal{R}$  be fixed and let $\vec{p}_R$ be the unique
point of $\mathcal{V}_{n,\sigma}$ contained in $R$. We can use
Bessel's inequality to bound the inner sum:
\begin{equation}
 \sum_{R^\prime \subseteq R} \frac {\ip  C_{\mathcal{V}_{n,\sigma}}  ,h _{R^\prime }, ^2  }
 {\lvert  R^\prime \rvert }\leq \norm  \mathbf{1}_{(\vec{p}_R, \vec{1}]} . L^2(R).^2 \leq  |R| \,
 .
\end{equation}
 Thus, by \eqref{e.b1}
\begin{align*}
 \lvert  U\rvert ^{-1}  \sum _{|R|<2^{-n}}  \frac {\ip  C_{\mathcal{V}_{n,\sigma}}  ,h _{R }, ^2  } {\lvert  R \rvert } \lesssim
  \lvert  U\rvert ^{-1} \sum_{R\in \mathcal{R}} \lvert R \vert \lesssim
  n.
\end{align*}

The proof is finished, since we have shown that for any measurable
set $U\subset [0,1]^2$
\begin{align*}
 \Bigg( \lvert  U\rvert ^{-1} \sum _{\substack{R \in \mathcal D ^2
\\ R\subset U}}
\frac {\ip D_N ,h _R , ^2  } {\lvert  R\rvert } \Bigg)^\frac12
\lesssim n^\frac12 \simeq \sqrt{\log N}.
\end{align*}
\end{proof}
%%%%%%%%%%%%%%%%%%%%%%%%%%%%%% PROOF PROOF PROOF
%%%%%%%%%%%%%%%%%%%%%%%%%%%%%% PROOF PROOF PROOF

%%%%%%%%%%%%%%%%%%%%%%%%%%%%%% SECTION SECTION SECTION
%%%%%%%%%%%%%%%%%%%%%%%%%%%%%% SECTION SECTION SECTION

\section{The $\textnormal{exp}(L^\alpha)$ Estimates for the Discrepancy Function.}

\subsection{Lower bound: The Proof of Theorem~\ref{t.lower}} %\label{s.}

The proof is by way of duality and is very similar to Hal{\'a}sz's proof
 \cite{MR637361} of Schmidt's Theorem, see  \eqref{e.schmidt}. Fix the
point distribution $ \mathcal A_N \subset [0,1] ^2 $. Set $ 2N<2
^{n} \le 4N$, so that $ n \simeq \log N$. Proposition~\ref{p.rvec}
provides us with $ \mathsf  r$ functions $ f _{\vec r} $ for $\vec
r\in \mathbb H _n ^{2}$. Let $ \mathbb G _N ^{2} \subset \mathbb H
_N ^{2}$ be those elements of $ \mathbb H _N ^{2}$ whose first
coordinate is a multiple of a sufficiently large integer $ a$. We
construct the following  functions:
\begin{equation} \label{e.Psi}
\Psi \coloneqq \prod _{\vec r \in \mathbb G _N ^2 } (1+ f _{\vec
r}), \qquad \quad \widetilde \Psi \coloneqq \Psi -1 .
\end{equation}
The `product rule' \ref{p.productRule} easily implies that $ \Psi $
is a positive function of $L^1$ norm one. In fact, letting $
g=\sharp \, \mathbb G _n  ^2 $, it is clear that
\begin{equation*}
\Psi = 2 ^{g} \mathbf 1_{E}\,, \qquad \mathbb P (E)= 2 ^{-g}\,.
\end{equation*}
Therefore, by Proposition~\ref{p.indicator},
\begin{equation*}
\norm \widetilde \Psi . L (\log L) ^{1/\alpha }.  \simeq g
^{1/\alpha } \simeq n ^{1/\alpha }\,.
\end{equation*}
The fact that $\ip D_N, \widetilde \Psi, \gtrsim n$ is well-known
\cite{MR637361}, \cite{MR1697825}. In fact, if we expand
\begin{align*}
\widetilde \Psi &= \sum _{k=1} ^{g} \Psi _{k} \,,
\\
\Psi _k & = \sum _{ \{\vec r_1 ,\dotsc, \vec r_k\} \subset \mathbb G
_n ^2 } \prod _{\ell =1} ^{k} f _{\vec r_\ell }\, ,
\end{align*}
then, using the `product rule' \ref{p.productRule}, it is not hard
to see that we have
\begin{equation}\label{e.a}
\ip D_N , \Psi _1 , \gtrsim  g \gtrsim \frac n a \, ,
\end{equation}
and  the other, higher order terms can be summed up, using
Propositions \ref{p.>n} and \ref{p.Not2Many}, to give a much smaller
estimate for $ a$ sufficiently large.

Thus, we can estimate
\begin{equation*}
n \lesssim  \ip D_N, \widetilde \Psi , \lesssim \norm D_N .
\operatorname {exp} (L ^{\alpha }). \cdot n ^{1/\alpha }\, ,
\end{equation*}
and so Theorem~\ref{t.lower} holds.

\subsection{Upper bound: The Proof of Theorem \ref{t.vdc} in the case that $ N=2 ^{n}$.}

In this section we shall obtain the upper bound of the
$\operatorname{exp} (L^2)$ norm of the discrepancy of the
digit-scrambled van der Corput set. We shall consider the case of $ N= 2 ^{n}$, leaving the 
general case to later. Lemma \ref{l.integral} tells us
that we should choose  $\mathcal V_{n,\sigma}$ with half the digits
`scrambled', i.e. $\sum_{i=1}^n \d_i (\sigma) = \lfloor n/2 \rfloor$
-- this will be the only restriction on $\sigma$ and for simplicity
we shall assume that $n$ is even.
%In view of Proposition \ref{p.comparable}, it would
%suffice to esttimate the ...
We expand $D_N$ in the Haar series and  break the expansion into
several parts (in view of our choice of $\sigma$, $h^{1,1}$ does not
play a role in the expansion):
\begin{align}
\nonumber D_N & = \sum_{R\in \mathcal D^2} \frac{\ip D_N, h_R,}{|R|}h_R
 + \sum_{R=I \times [0,1]} \frac{\ip D_N, h^{0,1}_R,}{|R|}h^{0,1}_R
 + \sum_{R= [0,1]\times I} \frac{\ip D_N, h^{1,0}_R,}{|R|}h^{1,0}_R\\
\label{e.hexp}& = \sum_{R:|R|> 2^{-n}} \frac{\ip D_N, h_R,}{|R|}h_R
+ \sum_{R:|R|\le 2^{-n}} \frac{\ip C_N, h_R,}{|R|}h_R  -
\sum_{R:|R|\le 2^{-n}} \frac{\ip L_N, h_R ,}{|R|}h_R\\
\label{e.hexp1}& \,\,\, + \sum_{R=I \times [0,1]} \frac{\ip D_N,
h^{0,1}_R,}{|R|}h^{0,1}_R
 + \sum_{R= [0,1]\times I} \frac{\ip D_N, h^{1,0}_R,}{|R|}h^{1,0}_R
\end{align}

For the first sum in the expansion \eqref{e.hexp} above we have:
\begin{align*}
\nonumber \NOrm \sum_{R:|R|> 2^{-n}} \frac{\ip D_N, h_R,}{|R|}h_R .
\operatorname{exp}(L^2). & \le \sum_{k=0}^{n-1} \NORm \sum_{R:|R|=
2^{-k}} \frac{\ip D_N, h_R,}{|R|}h_R . \operatorname{exp}(L^2).\\
\nonumber &\lesssim \sum_{k=0}^{n-1} \NORM \Bigg( \sum_{R:|R|=
2^{-k}} \frac{\ip D_N,
h_R, ^2}{|R|^2} {\mathbf 1}_R \Bigg)^{\frac12}.\infty.\\
 & \lesssim \sum_{k=0}^{n-1} \frac1{N} \cdot
\sqrt{k+1} \cdot 2^k \approx \sqrt{n},
\end{align*}
where we have used the hyperbolic version of the Chang-Wilson-Wolff
inequality (Theorem \ref{t.one}), the estimate of the Haar
coefficients of $D_N$ (Lemma \ref{l.haarcoeffs}), and the fact that
each point in $[0,1]^2$ lives in $k+1$ dyadic rectangles of volume
$2^{-k}$.

The last sum in \eqref{e.hexp} is easy to estimate. Since $\ip L_N,
h_R, = 4^{-d} N |R|^2$, we have:
\begin{align}
\nonumber \NOrm \sum_{R:|R|\le 2^{-n}} \frac{\ip L_N, h_R,}{|R|}h_R
. \operatorname{exp}(L^2). & \le 4^{-d} \sum_{k=n}^{\infty} \NOrm
\sum_{R:|R|=
2^{-k}} N 2^{-k} h_R . \operatorname{exp}(L^2).\\
\nonumber &\lesssim N \sum_{k=n}^\infty 2^{-k} \NORm \bigg(
\sum_{R:|R|= 2^{-k}}  {\mathbf 1}_R \bigg)^{\frac12}.\infty.\\
\label{e.exp3} & \lesssim N \sum_{k=n}^{\infty} \sqrt{k+1} \cdot
2^{-k} \approx \sqrt{n},
\end{align}
where we have once again applied  Theorem \ref{t.one}.

The second sum in \eqref{e.hexp} is the hardest. We consider
rectangles $R$ of volume $|R|\le 2^{-n}$. Recall that, in order for
$\ip C_N, h_R,$ to be non-zero, $R$ must contain points of $\mathcal
V_{n,\sigma}$ in the interior. The structure of the van der Corput
set then implies that we must at least have $|R_1|,|R_2| \ge
2^{-n}$. For each such rectangle $R$, one can find a unique
`parent': a dyadic rectangle $\widetilde{R} \subset [0,1]^2$ with
$|\widetilde{R}|= 2^{-n}$, $\widetilde{R}_1=R_1$, and $R\subset
\widetilde{R}$. We can now write
\begin{align}\label{e.e1}
\NOrm \sum_{R: |R|< 2^{-n}} \frac{\ip C_N, h_R,}{|R|} h_R . p. =
\NORM \sum_{k=0}^{n} \, \sum_{\substack{\widetilde{R}:\,
|\widetilde{R}|=2^{-n}\\ |\widetilde{R}_1|=2^{-k}}} \,\sum_{\substack{R\subset
\widetilde{R}\\ R_1=\widetilde{R}_1}} \frac{\ip C_N, h_R,}{|R|} h_R.
p.
\end{align}
A given rectangle $\widetilde{R}$ as above contains precisely one
point $(p_1,p_2)$ from the set $\mathcal V_{n,\sigma}$. Thus,
\begin{align}\label{e.e2}
\sum_{\substack{R\subset \widetilde{R}\\ R_1=\widetilde{R}_1}} \frac{\ip C_N,
h_R,}{|R|} \, h_R (x_1, x_2) =  C_{\widetilde{R}}(x_2) \frac{\ip
h_{\widetilde{R}_1}, \mathbf 1_{[p_1,1]},}{|\widetilde{R}_1|} \,
h_{\widetilde{R}_1} (x_1),
\end{align}
where
\begin{align}\label{e.e3}
C_{\widetilde{R}}(x_2) =\begin{cases} \sum_{I\subset \widetilde{R}_2} \frac{\ip
h_{I}, \mathbf 1_{[p_2,1]},}{|I|} \, h_{I} (x_2)
\nonumber  = \mathbf 1_{[p_2,1]} (x_2) - \int_{\widetilde{R}_2}
\mathbf 1_{[p_2,1]} (x) dx / |\widetilde{R}_2|,\ x_2\in \widetilde{R}_2,\\ \\
0, \ 	x_2 \not\in \widetilde{R}_2.
\end{cases}
\end{align}
In any case, we have $|C_{\widetilde{R}}(x_2)|\le 2$. Now we fix
$x_2 \in [0,1]$. For fixed $x_2$ and $\widetilde{R}_1$, there is a
unique $\widetilde{R}$ such that the sum in \eqref{e.e2} is
non-zero. Thus, using \eqref{e.e1}
\begin{align*}
\sum_{R: |R|\ge 2^{-n}} \frac{\ip C_N, h_R,}{|R|} \, h_R (x_1, x_2)&= \sum_{k=0}^n \,\, \sum_{\widetilde{R}_1:\,
|\widetilde{R}_1|=2^{-k}} \frac{C_{\widetilde{R}}(x_2) \ip
h_{\widetilde{R}_1}, \mathbf
1_{[p_1,1]},}{|\widetilde{R}_1|} \, h_{\widetilde{R}_1} (x_1)\\
&=\sum_{k=0}^n \,\, \sum_{\widetilde{R}_1:\,
|\widetilde{R}_1|=2^{-k}}
\frac{\alpha_{\widetilde{R}_1}(x_2)}{|\widetilde{R}_1|} \,
h_{\widetilde{R}_1} (x_1),
\end{align*}
where the Haar coefficient $\alpha_{\widetilde{R}_1}(x_2)$ satisfies
$|\alpha_{\widetilde{R}_1}(x_2)|\lesssim |\widetilde{R}_1|$. Next,
we apply the one-dimensional Littlewood-Paley inequality in the
variable $x_1$:
\begin{align*}
\NOrm \sum_{R: |R|\ge 2^{-n}} \frac{\ip C_N, h_R,}{|R|} h_R .
L^p(x_1). \lesssim p^\frac12 \NORM \, \Bigg(
\sum_{\widetilde{R}_1:\, |\widetilde{R}_1|\ge 2^{-n}}
\frac{|\alpha_{\widetilde{R}_1}(x_2)|^2}{|\widetilde{R}_1|^2} \,
\mathbf 1_{\widetilde{R}_1} \Bigg)^\frac12 .L^p (x_1). \le
p^\frac12 n^\frac12 .
\end{align*}
We now integrate this estimate in $x_2$ to obtain
\begin{align*}
\NOrm \sum_{R: |R|\ge 2^{-n}} \frac{\ip C_N, h_R,}{|R|} h_R . p.
\lesssim p^\frac12 n^\frac12,
\end{align*}
and thus
\begin{align*}
\NOrm \sum_{R: |R|\ge 2^{-n}} \frac{\ip C_N, h_R,}{|R|} h_R .
\operatorname{exp}(L^2). \lesssim  n^\frac12,
\end{align*}
in view of Proposition \ref{p.comparable}. Thus, we have estimated
the $\operatorname{exp}(L^2)$ norms of all the terms in
\eqref{e.hexp} by $n^\frac12$. The estimates for $(0,1)$ and $(1,0)$
Haars in \eqref{e.hexp1} can be easily incorporated, invoking
similar one-dimensional arguments and  Lemma \ref{l.semihaarcoeffs}.
We skip these computations for the sake of brevity. We thus arrive
to
\begin{equation}
\Norm D_N. \operatorname{exp}(L^2). \lesssim \sqrt{n} \approx
\sqrt{\log N}.
\end{equation}
Proposition \ref{p.AA} and inequality \eqref{e.schmidtsharp} finish
the proof of Theorem \ref{t.vdc} for all $\alpha \ge 2$.

%%%%%%%%%%%%%%%%%%%%%%%%%%%%%% SUBSECTION SUBSECTION SUBSECTION SUBSECTION
 %%%%%%%%%%%%%%%%%%%%%%%%%%%%%% SUBSECTION SUBSECTION SUBSECTION SUBSECTION 
\subsection{Upper bound: The Proof of Theorem \ref{t.vdc} in the General Case.}%\label{ss.}

We use a standard argument to generalize the previous proof to the case of arbitrary $ N$. 
Fix $ 2 ^{n-1}< N < N' \coloneqq 2 ^{n}$.  Set $ \tfrac 12 <t=N 2 ^{-n} + 2 ^{-n-1}<1$.  Consider the  following function 
\begin{equation*}
\Delta _{N}  (x_1, x_2) \coloneqq  D_{N'} (t x _1, x_2)-\tfrac 12 x_1 \cdot x_2\,, \qquad (x_1,x_2)\in [0,1] ^2  \,.  
\end{equation*}
Here, $ D _{N'}$ is the Discrepancy Function of  a shifted 
van der Corput set $ \mathcal V _{n, \sigma }$.    (The `$ -\tfrac 12 x_1 \cdot x_2$' above arises from the 
precise definition of the van der Corput set.)

The observation is that $ \Delta _N$ is in fact the Discrepancy Function of the set of points 
$  \{ v _{n,\sigma } (\tau ) \;:\; \tau = 0, 1 ,\dotsc, N\}$, where this notation is given in Definition~\ref{d.vdc}. 
For the linear part of the Discrepancy Function, note that 
\begin{equation*}
N' (t x_1) \cdot x_2 -\tfrac 12 x_1 \cdot x_2= N x_1 \cdot x_2 \,. 
\end{equation*}
And for the counting part, note that $ \mathbf 1_{ [ v _{n,\sigma } (\tau ),1) } (t x_1, x_2)$, 
restricted to $ [0,1] ^2 $ will be the indicator of a rectangle with one corner anchored at the upper right 
hand corner. Moreover, it will 
will be 
identically zero on $ [0,1] ^2 $ iff $ N<\tau \le N'$. Thus, 
$ \Delta _{N}   $ is a Discrepancy Function.

So it suffices for us to estimate the $ \operatorname {exp}(L ^{\alpha })$ norm of $ \Delta _N$.  
But this is straight forward.  
\begin{align*}
\norm \Delta _N . \operatorname {exp}(L ^{\alpha }).  
& \le  1+ \norm D _{N'} (t x_1, x_2)  . \operatorname {exp}(L ^{\alpha }).  
\\
& \le 1 + t ^{-1} \norm D _{N'} ( x_1, x_2)  . \operatorname {exp}(L ^{\alpha }). 
 \lesssim (\log N) ^{1/\alpha }\,, \qquad 2\le \alpha < \infty \,. 
\end{align*}

%%%%%%%%%%%%%%%%%%%%%%%%%%%%%% REMARK REMARK REMARK
\begin{remark}\label{r.bmo} We make a final remark on the other upper bound of the dyadic $ BMO$ estimate 
of the digit-scrambled van der Corput set in Theorem~\ref{t.bmo}.  It is natural to guess that this 
estimate should hold for all $ N$, and for $ BMO$.  A natural way to prove this is via the 
approach developed in \cites{MR2400405,809.3288}, but carrying out this argument is not completely straight forward.

\end{remark}
%%%%%%%%%%%%%%%%%%%%%%%%%%%%%% REMARK REMARK REMARK


\begin{bibsection}
 \begin{biblist}



\bib{MR1032337}{article}{
    author={Beck, J{\'o}zsef},
     title={A two-dimensional van Aardenne-Ehrenfest theorem in
            irregularities of distribution},
   journal={Compositio Math.},
    volume={72},
      date={1989},
    number={3},
     pages={269\ndash 339},
      issn={0010-437X},
    review={MR1032337 (91f:11054)},
}




\bib{MR903025}{book}{
    author={Beck, J{\'o}zsef},
    author={Chen, William W. L.},
     title={Irregularities of distribution},
    series={Cambridge Tracts in Mathematics},
    volume={89},
 publisher={Cambridge University Press},
     place={Cambridge},
      date={1987},
     pages={xiv+294},
      isbn={0-521-30792-9},
    review={MR903025 (88m:11061)},
}


% \bib{MR539351}{article}{
%    author={Bernard, Alain},
%    title={Espaces $H\sp{1}$ de martingales \`a deux indices. Dualit\'e avec
%    les martingales de type ``BMO''},
%    language={French, with English summary},
%    journal={Bull. Sci. Math. (2)},
%    volume={103},
%    date={1979},
%    number={3},
%    pages={297--303},
%    issn={0007-4497},
%    review={\MR{539351 (82d:60092)}},
% }



\bib{math.CA/0609815}{article}{
   author={Bilyk, Dmitriy},
   author={Lacey, Michael T.},
   title={On the small ball inequality in three dimensions},
   journal={Duke Math. J.},
   volume={143},
   date={2008},
   number={1},
   pages={81--115},
   issn={0012-7094},
   review={\MR{2414745}},
   eprint={	arXiv:math.CA/060981d},
}


\bib{0705.4619}{article}{
   author={Bilyk, Dmitriy},
   author={Lacey, Michael T.},
   author={Vagharshakyan, Armen},
   title={On the small ball inequality in all dimensions},
   journal={J. Funct. Anal.},
   volume={254},
   date={2008},
   number={9},
   pages={2470--2502},
   issn={0022-1236},
   review={\MR{2409170}},
   eprint={arXiv:0705.4619},
}
		



\bib{MR584078}{article}{
   author={Chang, Sun-Yung A.},
   author={Fefferman, Robert},
   title={A continuous version of duality of $H\sp{1}$ with BMO on the
   bidisc},
   journal={Ann. of Math. (2)},
   volume={112},
   date={1980},
   number={1},
   pages={179--201},
   issn={0003-486X},
   review={\MR{584078 (82a:32009)}},
}


\bib{cf1}{article}{
    author={Chang, Sun-Yung A.},
    author={Fefferman, Robert},
     title={Some recent developments in Fourier analysis and $H\sp p$-theory
            on product domains},
   journal={Bull. Amer. Math. Soc. (N.S.)},
    volume={12},
      date={1985},
    number={1},
     pages={1\ndash 43},
      issn={0273-0979},
    review={MR 86g:42038},
}

\bib{MR800004}{article}{
    author={Chang, S.-Y. A.},
    author={Wilson, J. M.},
    author={Wolff, T. H.},
     title={Some weighted norm inequalities concerning the Schr\"odinger
            operators},
   journal={Comment. Math. Helv.},
    volume={60},
      date={1985},
    number={2},
     pages={217\ndash 246},
      issn={0010-2571},
    review={MR800004 (87d:42027)},
}


\bib{MR610701}{article}{
   author={Chen, W. W. L.},
   title={On irregularities of distribution},
   journal={Mathematika},
   volume={27},
   date={1980},
   number={2},
   pages={153--170 (1981)},
   issn={0025-5793},
   review={\MR{610701 (82i:10044)}},
}

\bib{MR711520}{article}{
   author={Chen, W. W. L.},
   title={On irregularities of distribution. II},
   journal={Quart. J. Math. Oxford Ser. (2)},
   volume={34},
   date={1983},
   number={135},
   pages={257--279},
   issn={0033-5606},
   review={\MR{711520 (85c:11065)}},
}

\bib{Cor35}{article}{
    author={van der Corput, J. G.},
     title={Verteilungsfunktionen I},
   journal={Akad. Wetensch. Amdterdam, Proc.},
    volume={38},
      date={1935},
     pages={813\ndash 821},

}

 \bib{2000b:60195}{article}{
    author={Dunker, Thomas},
    author={K{\"u}hn, Thomas},
    author={Lifshits, Mikhail},
    author={Linde, Werner},
     title={Metric entropy of the integration operator and small ball
            probabilities for the Brownian sheet},
  language={English, with English and French summaries},
   journal={C. R. Acad. Sci. Paris S\'er. I Math.},
    volume={326},
      date={1998},
    number={3},
     pages={347\ndash 352},
      issn={0764-4442},
    review={MR2000b:60195},
}




% \bib{atomic}{article}{
%    author={Fefferman, R.},
%    title={The atomic decomposition of $H\sp 1$ in product spaces},
%    journal={Adv. in Math.},
%    volume={55},
%    date={1985},
%    number={1},
%    pages={90--100},
%    issn={0001-8708},
%    review={\MR{772072 (86e:42034)}},
% }


\bib{MR1439553}{article}{
    author={Fefferman, R.},
    author={Pipher, J.},
     title={Multiparameter operators and sharp weighted inequalities},
   journal={Amer. J. Math.},
    volume={119},
      date={1997},
    number={2},
     pages={337\ndash 369},
      issn={0002-9327},
    review={MR1439553 (98b:42027)},
}

\bib{MR637361}{article}{
   author={Hal{\'a}sz, G.},
   title={On Roth's method in the theory of irregularities of point
   distributions},
   conference={
      title={Recent progress in analytic number theory, Vol. 2},
      address={Durham},
      date={1979},
   },
   book={
      publisher={Academic Press},
      place={London},
   },
   date={1981},
   pages={79--94},
   review={\MR{637361 (83e:10072)}},
}



\bib {HAZ}{article}{
AUTHOR = {Halton, J. H.},
author={ Zaremba, S. K.},
     TITLE = {The extreme and {$L\sp{2}$} discrepancies of some plane sets},
   JOURNAL = {Monatsh. Math.},
    VOLUME = {73},
      YEAR = {1969},
     PAGES = {316--328},
   MRCLASS = {10.33},
  MRNUMBER = {MR0252329 (40 \#5550)},
MRREVIEWER = {O. P. Stackelberg},
}

\bib {KrPil}{article}{
AUTHOR = {Kritzer, P.}, author={ Pillichshammer, F.},
     TITLE = {An exact formula for the $L_2$ discrepancy of the shifted Hammersley point set},
   JOURNAL = {Uniform Distribution Theory},
    VOLUME = {1},
      YEAR = {2006},
    number = {1}
     PAGES = {1-13},
      }


  \bib{math.NT/0609817}{article}{
    title={{On the Discrepancy Function in Arbitrary Dimension, Close to
        $ L ^{1}$}},
    author={Lacey, Michael T },
    eprint={arXiv:math.NT/060981d},
    journal={to appear in Analysis Mathematica},
    date={2006}
}



\bib{MR0500056}{book}{
   author={Lindenstrauss, Joram},
   author={Tzafriri, Lior},
   title={Classical Banach spaces. I},
   note={Sequence spaces;
   Ergebnisse der Mathematik und ihrer Grenzgebiete, Vol. 92},
   publisher={Springer-Verlag},
   place={Berlin},
   date={1977},
   pages={xiii+188},
   isbn={3-540-08072-4},
   review={\MR{0500056 (58 \#17766)}},
}


\bib{MR1697825}{book}{
   author={Matou{\v{s}}ek, Ji{\v{r}}\'\i},
   title={Geometric discrepancy},
   series={Algorithms and Combinatorics},
   volume={18},
   note={An illustrated guide},
   publisher={Springer-Verlag},
   place={Berlin},
   date={1999},
   pages={xii+288},
   isbn={3-540-65528-X},
   review={\MR{1697825 (2001a:11135)}},
}




\bib{MR850744}{article}{
    author={Pipher, Jill},
     title={Bounded double square functions},
  language={English, with French summary},
   journal={Ann. Inst. Fourier (Grenoble)},
    volume={36},
      date={1986},
    number={2},
     pages={69\ndash 82},
      issn={0373-0956},
    review={MR850744 (88h:42021)},
}

\bib{MR2400405}{article}{
   author={Pipher, Jill},
   author={Ward, Lesley A.},
   title={BMO from dyadic BMO on the bidisc},
   journal={J. Lond. Math. Soc. (2)},
   volume={77},
   date={2008},
   number={2},
   pages={524--544},
   issn={0024-6107},
   review={\MR{2400405}},
}



\bib{MR0066435}{article}{
   author={Roth, K. F.},
   title={On irregularities of distribution},
   journal={Mathematika},
   volume={1},
   date={1954},
   pages={73--79},
   issn={0025-5793},
   review={\MR{0066435 (16,575c)}},
}






\bib{MR0319933}{article}{
   author={Schmidt, Wolfgang M.},
   title={Irregularities of distribution. VII},
   journal={Acta Arith.},
   volume={21},
   date={1972},
   pages={45--50},
   issn={0065-1036},
   review={\MR{0319933 (47 \#8474)}},
}




\bib{MR554923}{book}{
   author={Schmidt, Wolfgang M.},
   title={Lectures on irregularities of distribution},
   series={Tata Institute of Fundamental Research Lectures on Mathematics
   and Physics},
   volume={56},
   publisher={Tata Institute of Fundamental Research},
   place={Bombay},
   date={1977},
   pages={v+128},
   review={\MR{554923 (81d:10047)}},
}


\bib{MR0491574}{article}{
   author={Schmidt, Wolfgang M.},
   title={Irregularities of distribution. X},
   conference={
      title={Number theory and algebra},
   },
   book={
      publisher={Academic Press},
      place={New York},
   },
   date={1977},
   pages={311--329},
   review={\MR{0491574 (58 \#10803)}},
}


\bib{MR95k:60049}{article}{
    author={Talagrand, Michel},
     title={The small ball problem for the Brownian sheet},
   journal={Ann. Probab.},
    volume={22},
      date={1994},
    number={3},
     pages={1331\ndash 1354},
      issn={0091-1798},
    review={MR 95k:60049},
}

\bib{MR96c:41052}{article}{
    author={Temlyakov, V. N.},
     title={An inequality for trigonometric polynomials and its application
            for estimating the entropy numbers},
   journal={J. Complexity},
    volume={11},
      date={1995},
    number={2},
     pages={293\ndash 307},
      issn={0885-064X},
    review={MR 96c:41052},
}

   
\bib{809.3288}{article}{
  author={Sergei Treil},
  title={$H^1$ and dyadic $H^1$},
  date={2008},
  eprint={http://arxiv.org/abs/0809.3288},
}




\bib{MR1018577}{article}{
   author={Wang, Gang},
   title={Sharp square-function inequalities for conditionally symmetric
   martingales},
   journal={Trans. Amer. Math. Soc.},
   volume={328},
   date={1991},
   number={1},
   pages={393--419},
   issn={0002-9947},
   review={\MR{1018577 (92c:60067)}},
}





  \end{biblist}
 \end{bibsection}
 \end{document}